\documentclass{article}

\usepackage[a4paper]{geometry}
\usepackage{amsmath}
\usepackage{amssymb}
\usepackage{amsthm}
\usepackage{multirow}
\usepackage{enumerate}

\allowdisplaybreaks[1]

\newtheorem{theorem}{Theorem}
\newtheorem{corollary}[theorem]{Corollary}
\newtheorem{lemma}[theorem]{Lemma}
\newtheorem{proposition}[theorem]{Proposition}

\newtheorem{example}[theorem]{Example}

\numberwithin{theorem}{section}
\numberwithin{equation}{section}
\numberwithin{table}{section}

\newtheorem{remark}[theorem]{Remark}

\newcommand{\bb}{\mathbb}
\newcommand{\Z}{\bb{Z}}

\newcommand{\Q}{\bb{Q}}
\newcommand{\cQ}{\mathcal{Q}}

\newcommand{\ac}{a_{\chi}}
\newcommand{\al}{\alpha}
\newcommand{\bc}{b_{\chi}}

\newcommand{\la}{\lambda}
\newcommand{\ze}{\zeta}

\newcommand{\sig}{\sigma}

\newcommand{\pr}{\prime}
\newcommand{\sm}{\setminus}
\newcommand{\ol}{\overline}
\newcommand{\lan}{\langle}
\newcommand{\ran}{\rangle}

\newcommand{\wti}{\widetilde}
\newcommand{\wh}{\widehat}

\newcommand{\lc}{\lceil}
\newcommand{\rc}{\rceil}

\newcommand{\F}{\mathbb{F}}

\newcommand{\alj}{\al_{j,\chi}}
\newcommand{\alp}{{\al^+_\chi}}
\newcommand{\alm}{{\al^-_\chi}}
\newcommand{\lpc}{{\ell^+_\chi}}
\newcommand{\lmc}{{\ell^-_\chi}}

\newcommand{\zpe}{\zeta_{p^e}}
\newcommand{\Ft}{\F_3}
\newcommand{\Ftf}{\F_3^5}
\newcommand{\PG}{\text{PG}}
\newcommand{\GL}{\text{GL}}
\newcommand{\PGL}{\text{PGL}}

\long\def\symbolfootnote[#1]#2{\begingroup%
\def\thefootnote{\fnsymbol{footnote}}\footnote[#1]{#2}\endgroup}

\begin{document}

\title{Construction and nonexistence of strong external difference families}
\author{Jonathan Jedwab \and Shuxing Li}
\date{20 January 2017 (revised 15 November 2017)}
\maketitle

\symbolfootnote[0]{
Department of Mathematics, Simon Fraser University, 8888 University Drive, Burnaby BC V5A 1S6, Canada.
\par
J.~Jedwab is supported by NSERC.
\par
Email: {\tt jed@sfu.ca}, {\tt shuxing\_li@sfu.ca}
}

\begin{abstract}
Strong external difference families (SEDFs) were introduced by Paterson and Stinson as a more restrictive version of external difference families. SEDFs can be used to produce optimal strong algebraic manipulation detection codes.
We characterize the parameters $(v, m, k, \la)$ of a nontrivial SEDF that is near-complete (satisfying $v=km+1$). We construct the first known nontrivial example of a $(v, m, k, \la)$ SEDF having $m > 2$. The parameters of this example are $(243,11,22,20)$, giving a near-complete SEDF, and its group is $\Z_3^5$.
We provide a comprehensive framework for the study of SEDFs using character theory and algebraic number theory, showing that the cases $m=2$ and $m>2$ are fundamentally different. We prove a range of nonexistence results, greatly narrowing the scope of possible parameters of SEDFs.

\end{abstract}

\section{Introduction}
\label{sec-intro}
Let $G$ be an abelian group of order $v$ with identity~$1$. We shall work in the setting of the group ring $\Z[G]$: given a subset $D$ of $G$, we write the group ring element $\sum_{d \in D} d$ as $D$ (by a standard abuse of notation), and the group ring element $\sum_{d \in D} d^{-1}$ as $D^{(-1)}$.
Let $D_1, D_2, \dots, D_m$ be mutually disjoint $k$-subsets of $G$, where $m \ge 2$, and let $\la$ be a positive integer. Then $\{D_1, D_2, \dots, D_m\}$ is a $(v, m, k, \la)$-\emph{external difference family} in $G$ if
\begin{equation}
\sum_{\substack{1\le i,j \le m \\ i \ne j}} D_jD_i^{(-1)}=\la(G-1) \quad \mbox{in $\Z[G]$},
\end{equation}
and is a $(v, m, k, \la)$-\emph{strong external difference family} (SEDF) in $G$ if
\begin{equation}\label{eqn-def}
D_j \sum_{\substack{1 \le i \le m \\ i \ne j}} D_i^{(-1)}=\la(G-1) \quad \mbox{in $\Z[G]$ for each $j$ satisfying $1 \le j \le m$}.
\end{equation}
The use of ``strong'' arises because a $(v,m,k,\la)$-SEDF is necessarily a $(v,m,k,m\la)$-external difference family.

External difference families have applications in authentication codes and secret sharing \cite{OKSS}. An external difference family in a cyclic group gives rise to difference systems of sets \cite{CD}, which can be applied to construct synchronization codes \cite{Lev}.
Paterson and Stinson \cite{PS} introduced SEDFs and showed how to produce optimal strong algebraic manipulation detection codes from them.
Algebraic manipulation detection codes have many applications, including robust secret sharing schemes, secure multiparty computation, and non-malleable codes \cite{CDF+,CFP,CPX}.
A succession of recent papers has demonstrated that SEDFs are interesting combinatorial objects in their own right: see Proposition~\ref{prop-known} below for a summary of constructive results, Proposition~\ref{prop-smallla} for a characterization result, and Proposition~\ref{prop-nonexistence} for a selection of nonexistence results.

The parameters of a $(v,m,k,\la$)-SEDF satisfy the counting relation
\begin{equation}\label{eqn-counting}
k^2(m-1)=\la(v-1).
\end{equation}
A $(v,m,k,\la)$-SEDF is {\it trivial} if $k=1$; it follows from \eqref{eqn-counting} that the parameters of a trivial SEDF have the form $(v,v,1,1)$, and an SEDF with these parameters exists (trivially) in every group of order~$v$.
The following proposition describes the parameters and groups of the known nontrivial SEDFs, all of which satisfy $m=2$.

\begin{proposition}\label{prop-known}
A $(v,m,k,\la)$-SEDF exists in the group $G$ in each of the following cases:
\begin{enumerate}[(1)]
\item $(v,m,k,\la)=(k^2+1,2,k,1)$ and $G=\Z_{k^2+1}$ {\rm\cite[Example 2.2]{PS}}.
\item $(v,m,k,\la)=(v,2,\frac{v-1}{2},\frac{v-1}{4})$ and $v \equiv 1 \pmod4$, provided there exists a $(v,\frac{v-1}{2},\frac{v-5}{4},\frac{v-1}{4})$ partial difference set in $G$ {\rm\cite[Section 3]{DHM}}, {\rm\cite[Theorem 4.4]{HP}}.
\item $(v,m,k,\la)=(p,2,\frac{p-1}{4},\frac{p-1}{16})$ where $p=16t^2+1$ is a prime and $t$ is an integer, and $G=\Z_p$ {\rm \cite[Theorem 4.3]{BJWZ}}.
\item $(v,m,k,\la)=(p,2,\frac{p-1}{6},\frac{p-1}{36})$ where $p=108t^2+1$ is a prime and $t$ is an integer, and $G=\Z_p$ {\rm \cite[Theorem 4.6]{BJWZ}}.
\end{enumerate}
\end{proposition}

\begin{remark}
Proposition~\ref{prop-known}~(3) describes a construction which was presented in \cite[Theorem 4.3]{BJWZ} with a prime power $q$ in place of the prime $p$ and with $G=\F_q$ in place of $G=\Z_p$. However, if $q=16t^2+1$ is a prime power and $t$ is an integer, then $q$ must be a prime because Catalan's conjecture is known to hold~\cite{M}. 
\end{remark}

When $\la = 1$, the parameters of a nontrivial $(v,m,k,\la)$-SEDF have been characterized.

\begin{proposition}\label{prop-smallla}
A nontrivial $(v,m,k,1)$-SEDF exists if and only if $m=2$ and $v=k^2+1$ {\rm \cite[Theorem 2.3]{PS}}.
\end{proposition}

The following proposition describes parameter sets $(v,m,k,\la)$ for which a nontrivial SEDF is known not to exist in all groups of order~$v$.

\begin{proposition}\label{prop-nonexistence}
A nontrivial $(v,m,k,\la)$-SEDF does not exist in each of the following cases:
\begin{enumerate}[(1)]
\item $m \in \{3,4\}$ {\rm \cite[Theorems 3.3 and 3.6]{MS}}
\item $m>2$ and $v$ is prime {\rm \cite[Theorem 3.9]{MS}}
\item $m>2$ and $\la=2$ {\rm \cite[Corollary 3.2]{HP}}
\item $m>2$ and $\la>1$ and $\frac{\la(k-1)(m-2)}{(\la-1)k(m-1)}>1$ {\rm \cite[Theorem 3.5]{HP}}
\item $m>2$ and there is a prime $p$ dividing $v$ for which $\gcd(km,p)=1$ and $m \not\equiv 2 \pmod{p}$ {\rm \cite[Theorem 3.5]{BJWZ}}.
\item $\la \ge k$ {\rm \cite[Lemma 1.1]{BJWZ}}
\end{enumerate}
\end{proposition}

It is known \cite[Lemma 1.2]{MS} that if $v=km$, then an $(v,m,k,\la)$-SEDF is necessarily trivial. The same proof idea as in \cite{MS} gives the following generalization.

\begin{lemma}\label{lem-trivial}
Suppose there exists a $(v,m,k,\la)$-SEDF for which $\gcd(k,v-1)=1$. Then the SEDF is trivial.
\end{lemma}
\begin{proof}
The counting relation \eqref{eqn-counting} gives $m-1 = \frac{\la}{k^2} (v-1)$.
Since $\gcd(k,v-1)=1$, it follows that $\la/k^2$ is an integer and so $m-1 \ge v-1$. Since $v \ge km$, this implies that $k=1$.
\end{proof}

Lemma~\ref{lem-trivial} implies that the parameters of a nontrivial $(v,m,k,\la)$-SEDF $\{D_1, D_2, \dots, D_m\}$ satisfy $v > km$ and, by taking a translate of all the subsets $D_j$ if necessary, we may therefore assume that $1 \notin \bigcup_{j=1}^m D_j$. In the extremal case $v=km+1$, the subsets $D_1,D_2,\ldots,D_m$ partition the nonidentity elements of the group $G$ and (following \cite{DHM}) we call the SEDF \emph{near-complete}.

In this paper, we present constructive and nonexistence results for nontrivial SEDFs using character theory and algebraic number theory.
In Section~\ref{sec-char}, we give a character-theoretic framework for the study of SEDFs and demonstrate that the cases $m=2$ and $m>2$ are fundamentally different.
In Section~\ref{sec-near-complete}, we characterize the parameters of a nontrivial near-complete SEDF by establishing an equivalence with a collection of partial difference sets. In particular, we construct a near-complete $(243,11,22,20)$-SEDF in $\Z_3^5$ by reference to the point-orbits of the Mathieu group $M_{11}$ acting on the projective geometry $PG(4,3)$. This is the first known nontrivial example of an SEDF with $m>2$.
In Section~\ref{sec-exp-bound}, we use algebraic number theory to obtain an exponent bound on a group containing a SEDF and apply it to rule out various SEDFs with $m=2$, leaving only 5 open cases for the parameters of a $(v,m,k,\la)$ SEDF with $v \le 50$ and $m=2$.
In Section~\ref{sec-m>2} we obtain nonexistence results for SEDFs with $m>2$, introducing the ``simple character value property'' under which strong necessary conditions can be derived.
This leaves only 70 open cases for the parameters of a $(v,m,k,\la)$ SEDF with $v \le 10^4$ and $m>2$.

\section{A character-theoretic approach}
\label{sec-char}

Let $\wh{G}$ denote the character group of an abelian group~$G$, and let $\chi_0 \in \wh{G}$ be the principal character. Each character $\chi \in \wh{G}$ is extended linearly to the group ring $\Z[G]$. The following formula is a consequence of the orthogonality properties of characters.

\begin{proposition}[Fourier inversion formula] \label{prop-fourier}
Let $G$ be an abelian group and let $A=\sum_{g \in G} c_gg \in \Z[G]$. Then
$$
c_g=\frac{1}{|G|}\sum_{\chi \in \wh{G}} \chi(A)\ol{\chi(g)} \quad \mbox{for each $g \in G$}.
$$
\end{proposition}

Suppose $\{D_1,D_2,\dots,D_m\}$ is a nontrivial $(v,m,k,\la)$-SEDF in a group $G$, and write $D=\bigcup_{i=1}^m D_i$. Then (\ref{eqn-def}) is equivalent to
\[
D_j(D^{(-1)}-D_j^{(-1)})=\la(G-1) \quad \mbox{in $\Z[G]$ for each $j$ satisfying $1 \le j \le m$}.
\]
Apply a nonprincipal character $\chi \in \wh{G}$ to obtain
\begin{equation}\label{eqn-chardef}
\chi(D_j)\big(\ol{\chi(D)-\chi(D_j)}\big)=-\la \quad \mbox{for all nonprincipal $\chi \in \wh{G}$ and for each $j$}.
\end{equation}
Some basic restrictions were derived from (\ref{eqn-chardef}) in \cite[Lemma 3.1]{MS}. We now extend that analysis.

It follows from \eqref{eqn-chardef} that for each $j$ satisfying $1 \le j \le m$,
\begin{equation}\label{eqn-chiDlaiff}
|\chi(D_j)|^2=\la \quad \mbox{if and only if} \quad \chi(D) = 0.
\end{equation}
Define
\begin{align}
\wh{G}^0 &=\{\mbox{nonprincipal } \chi \in \wh{G} \mid \chi(D)=0\}, \label{eqn-G0defn} \\
\wh{G}^N &= \{\mbox{nonprincipal }\chi \in \wh{G} \mid \chi(D) \ne 0\}, \nonumber 
\end{align}
so that $\wh{G}$ may be partitioned (with respect to $D$) as the disjoint union $\{\chi_0\} \cup \wh{G}^0 \cup \wh{G}^N$.
We now show that the set $\wh{G}^N$ is non-empty.

\begin{lemma}[{\cite[Lemma 3.1 (d)]{MS}}]\label{lem-GNnonempty}
Suppose $\{D_1,D_2,\dots,D_m\}$ is a nontrivial $(v,m,k,\la)$-SEDF in a group $G$, and let $D=\bigcup_{i=1}^m D_i$. Then $|\wh{G}^N| > 0$.
\end{lemma}
\begin{proof}
Suppose, for a contradiction, that $\chi(D) = 0$ for each nonprincipal $\chi \in \wh{G}$. Write $D = \sum_{g\in G} c_g g$ in $\Z[G]$ and use Proposition~\ref{prop-fourier} to show that for each $g \in G$ we have
\[
c_g = \frac{1}{v}\chi_0(D)\ol{\chi_0(g)} = \frac{km}{v}.
\]
By Lemma~\ref{lem-trivial} we have $v > km$, giving the contradiction $0 < c_g < 1$.
\end{proof}

For each $\chi \in \wh{G}^N$, set $\al_{j,\chi}$ to be the real number $\frac{|\chi(D_j)|^2}{|\chi(D_j)|^2-\la}$. Then conjugate \eqref{eqn-chardef}, multiply both sides by $\chi(D_j)$, and rearrange to give
\begin{equation}\label{eqn-charval}
\chi(D_j)=\al_{j,\chi}\,\chi(D) \quad \mbox{for $\chi \in \wh{G}^N$}.
\end{equation}
Substitute for $\chi(D_j)$ from \eqref{eqn-charval} into \eqref{eqn-chardef} to obtain a quadratic equation in $\al_{j,\chi}$:
\begin{equation}\label{eqn-quadeqn}
\alj^2-\alj-\tfrac{\la}{|\chi(D)|^2}=0 \quad \mbox{for $\chi \in \wh{G}^N$}.
\end{equation}
The solutions of this equation are
\begin{equation}\label{eqn-sols}
\alp=\frac{1}{2}\Big(1+\sqrt{1+\tfrac{4\la}{|\chi(D)|^2}}\Big), \quad
\alm=\frac{1}{2}\Big(1-\sqrt{1+\tfrac{4\la}{|\chi(D)|^2}}\Big)
\quad \mbox{for $\chi \in \wh{G}^N$}.
\end{equation}
For each $\chi \in \wh{G}^N$, let $\lpc$ and $\lmc$ be the number of times $\alj$ takes the value $\alp$ and $\alm$, respectively, as $j$ ranges over $1 \le j \le m$. Using $\chi(D)=\sum_{j=1}^m\chi(D_j)$, we find from \eqref{eqn-charval} that
$$
\lpc \alp + \lmc \alm =1.
$$
Combine with the counting condition $\lpc+\lmc=m$ to determine $\lpc$ and $\lmc$ as
\begin{equation}\label{eqn-solsfre}
\lpc=\frac{m}{2}-\frac{m-2}{2\sqrt{1+\frac{4\la}{|\chi(D)|^2}}}, \quad
\lmc=\frac{m}{2}+\frac{m-2}{2\sqrt{1+\frac{4\la}{|\chi(D)|^2}}} \quad
\quad \mbox{for $\chi \in \wh{G}^N$}.
\end{equation}
In particular, $\lpc \ge \frac{m}{2}-\frac{m-2}{2} = 1$ and $\lmc \ge 1$, so the values $\alp$ and $\alm$ both occur as $j$ ranges over $\{1,2,\dots,m\}$. Therefore from \eqref{eqn-charval} we have
 \begin{equation}\label{eqn-chiDj}
\{\chi(D_j) \mid 1 \le j \le m\} =
\{\alp\,\chi(D), \alm\,\chi(D)\} \quad \mbox{for $\chi \in \wh{G}^N$}.
\end{equation}

The expressions \eqref{eqn-solsfre} illustrate a fundamental difference between the cases $m=2$ and $m>2$. When $m=2$, these expressions reduce to $\lpc = \lmc = 1$. But when $m>2$, we require $\sqrt{1+\frac{4\la}{|\chi(D)|^2}} \in \Q$ for each $\chi \in \wh{G}^N$ in order for $\lpc$ and $\lmc$ to be integers. We shall see in Section~\ref{sec-m>2} that this yields strong restrictions on the character values of $\chi(D)$ and $\chi(D_j)$ for SEDFs when $m>2$, which do not apply when $m=2$.

We conclude this section with a result required in Section~\ref{sec-near-complete}.

\begin{lemma}\label{lem-subgp}
Suppose $\{D_1,D_2,\ldots,D_m\}$ is a nontrivial $(v,m,k,\la)$-SEDF in a group~$G$, where $1 \notin \bigcup_{i=1}^m D_i$. Then, for each $j$,
neither $D_j \cup \{1\}$ nor $G \sm D_j$ is a subgroup of $G$.
\end{lemma}
\begin{proof}

Suppose, for a contradiction, that $D_j \cup \{1\}$ is a subgroup of $G$. Since $m \ge 2$, there exists a nonprincipal character $\chi$ of $G$ which is principal on $D_j \cup \{1\}$. Then (\ref{eqn-chardef}) gives $k(\ol{\chi(D)}-k)=-\la$, so that $\chi(D)=k-\frac{\la}{k}$ is a rational number. Since $\chi(D)$ is also an algebraic integer, $\la/k$ is an integer and therefore $\la \ge k$. This contradicts Proposition~\ref{prop-nonexistence}~(6).

Suppose, for a contradiction, that $G \sm D_j$ is a subgroup of $G$. Then $(v-k) \mid v$, and since $k > 1$ we have $k \ge \frac{v}{2}$. But
$v > km$ and $m \ge 2$ gives the contradiction $k < \frac{v}{2}$.
\end{proof}

\section{Near-complete SEDFs}
\label{sec-near-complete}
Let $D$ be a $k$-subset of a group $G$ of order~$v$, where $1 \notin D$.
The subset $D$ is a {\it $(v,k,\la,\mu)$ partial difference set} (PDS) in $G$ if
\begin{equation}\label{defn-PDS}
DD^{(-1)}=(k-\mu)\cdot 1+\la D+\mu(G-D) \quad \mbox{in $\Z[G]$}.
\end{equation}
(A slightly different definition, which we will not require, applies when $1 \in D$.)
The PDS $D$ is {\it regular} if $D = D^{(-1)}$, and is {\it trivial} if either $D \cup \{1\}$ or $G \sm D$ is a subgroup of~$G$.
In this section we prove the following result, which characterizes the parameters of a nontrivial near-complete $(v, m, k, \la)$-SEDF and provides the first known example of a nontrivial $(v, m, k, \la)$-SEDF having $m > 2$.

\begin{theorem}\label{thm-characterization}
Let $D_1, D_2, \dots, D_m$ partition the nonidentity elements of an abelian group $G$ of order $v = km+1$ into $m$ subsets each of size~$k > 1$. Then $\{D_1, D_2, \dots, D_m\}$ is a nontrivial near-complete $(v,m,k,\la)$-SEDF in $G$ if and only if either
\begin{enumerate}[(1)]
\item
$(v,m,k,\la)=(v,2,\frac{v-1}{2},\frac{v-1}{4})$
and $v \equiv 1 \pmod{4}$
and $D_1$ is a nontrivial regular $(v,\frac{v-1}{2},\frac{v-5}{4},\frac{v-1}{4})$-PDS in~$G$,
or
\item
$(v,m,k,\la)=(243,11,22,20)$
and each $D_j$ is a nontrivial regular $(243, 22, 1, 2)$-PDS in~$G$ for $1 \le j \le 11$.
\end{enumerate}
Furthermore, a $(243, 11, 22, 20)$-SEDF exists in~$\Z_3^5$.

\end{theorem}
The restriction of Theorem~\ref{thm-characterization} to the case $m=2$ is due to Huczynska and Paterson \cite[Theorem~4.6]{HP}, and also to Ding \cite[Proposition 2.1]{D} from the viewpoint of difference families. One direction of the case $m=2$, namely the construction of an SEDF from a PDS, was also proved in \cite[Section 3]{DHM}.
Necessary and sufficient conditions for the existence of a PDS with the parameters specified in (1) and (2) of Theorem~\ref{thm-characterization} are not known. However, sufficient conditions for the existence of a PDS with the parameters specified in (1) of Theorem~\ref{thm-characterization} (known as a Paley-type PDS) are known to include:
$G$ is elementary abelian and $v$ is a prime power congruent to 1 modulo 4~\cite{Paley};
$G = \Z_{p^r}^2$ for an odd prime $p$~\cite{LM};
and $G = \Z_3^2 \times \Z_{p}^{4r}$ for an odd prime $p$~\cite{Polhill}.
Necessary conditions for the existence of a PDS in an abelian group $G$ with the parameters specified in (2) of Theorem~\ref{thm-characterization} are that $G = \Z_3^5$, $\Z_3^3 \times \Z_9$, or $\Z_3 \times \Z_9^2$ \cite[Theorem 6.9]{Ma94}; existence is known for $G = \Z_3^5$ \cite{BLS}, \cite[Section 10]{CK}.

In order to establish Theorem~\ref{thm-characterization}, we make the following connection between a nontrivial near-complete SEDF and a collection of nontrivial regular PDSs.

\begin{lemma}\label{lem-PDS}
Let $D_1,D_2,\ldots,D_m$ partition the nonidentity elements of an abelian group~$G$ of order $v = km+1$ into $m$ subsets each of size~$k > 1$. Then $\{D_1,D_2,\ldots,D_m\}$ is a nontrivial near-complete $(v,m,k,\la)$-SEDF in $G$ if and only if each $D_j$ is a nontrivial regular $(v,k,k-\la-1,k-\la)$-PDS in $G$ for $1 \le j \le m$.
\end{lemma}
\begin{proof}
Since $D_1,D_2,\ldots,D_m$ is a partition of the nonidentity elements of $G$,
for each $j$ satisfying $1 \le j \le m$ we have $1 \notin D_j$ and
$$
\sum_{\substack{1 \le i \le m \\ i \ne j}} D_i=G-D_j-1,
$$
and therefore
\[
D_j \sum_{\substack{1 \le i \le m \\ i \ne j}}D_i^{(-1)}=D_j(G-D_j^{(-1)}-1)
= k G - D_jD_j^{(-1)} - D_j.
\]
It follows that $\{D_1,D_2,\ldots,D_m\}$ is a nontrivial near-complete $(v,m,k,\la)$-SEDF in $G$ if and only if, for each $j$,
$$
\la(G-1) = k G - D_jD_j^{(-1)} - D_j,
$$
which rearranges to
\begin{equation}\label{eqn-rearranges}
D_jD_j^{(-1)}=\la\cdot 1+(k-\la-1)D_j+(k-\la)(G-D_j).
\end{equation}
Equivalently, each $D_j$ is a $(v, k, k-\la-1, k-\la)$-PDS in $G$.

To complete the proof, we require that if $\{D_1,D_2,\ldots,D_m\}$ is a nontrivial near-complete $(v,m,k,\la)$-SEDF in $G$, then each PDS $D_j$ is nontrivial and regular.  Nontriviality of each $D_j$ is given by Lemma~\ref{lem-subgp}, and regularity by \cite[Proposition 1.2]{Ma94}.
\end{proof}

The parameters of the nontrivial regular PDSs specified in Lemma~\ref{lem-PDS} take the form $(v, k, \mu-1, \mu)$. The following result characterizes all such parameters when the group is abelian.

\begin{theorem}[{\cite{AJMP}; see also \cite[Theorem 13.1]{Ma94}}]
\label{thm-PDS}
Suppose there exists a nontrivial regular $(v,k,\mu-1,\mu)$-PDS in an abelian group. Then either
\begin{enumerate}[(1)]
\item
$(v,k,\mu-1,\mu)=(v,\frac{v-1}{2},\frac{v-5}{4},\frac{v-1}{4})$ and $v \equiv 1 \pmod{4}$, or
\item
$(v,k,\mu-1,\mu)=(243,22,1,2)$ or $(243,220,199,220)$.
\end{enumerate}
\end{theorem}

We can now give the structure of the proof of Theorem~\ref{thm-characterization}.
\begin{proof}[Proof of Theorem~\ref{thm-characterization}]
By Lemma~\ref{lem-PDS},
$\{D_1, D_2, \dots, D_m\}$ is a nontrivial near-complete $(v,m,k,\la)$-SEDF in $G$ if and only if
each $D_j$ is a nontrivial regular $(v, k, k-\la-1, k-\la)$-PDS in $G$ for $1 \le j \le m$.
Since $m = (v-1)/k$, by Theorem~\ref{thm-PDS} this holds if and only if either
\begin{enumerate}[(1)]
\item $(v,m,k,\la)=(v,2,\frac{v-1}{2},\frac{v-1}{4})$ and $v \equiv 1 \pmod{4}$ and each of $D_1, D_2$ is a nontrivial regular $(v, \frac{v-1}{2}, \frac{v-5}{4}, \frac{v-1}{4})$-PDS in $G$, or
\item $(v,m,k,\la)=(243,11,22,20)$ and each $D_j$ is a nontrivial regular $(243, 22, 1, 2)$-PDS in $G$ for $1 \le j \le 11$.
\end{enumerate}
For case (1), the desired result follows from the observation that if $D_1$ is a regular $(v, \frac{v-1}{2}, \frac{v-5}{4}, \frac{v-1}{4})$-PDS in $G$ that does not contain the identity, then so is $D_2 = G\setminus (D_1 \cup \{1\})$ \cite[Lemma~4.3]{HP}.

It remains to construct a $(243,11,22,20)$-SEDF in $\Z_3^5$, which is carried out below.
\end{proof}

In the rest of this section we shall construct a $(243,11,22,20)$-SEDF in $\Z_3^5$, regarded as the additive group of~$\Ftf$. By Lemma~\ref{lem-PDS}, this is equivalent to partitioning the nonzero elements of $\Ftf$ into 11 subsets, each of which is a nontrivial regular $(243,22,1,2)$-PDS in the additive group of~$\Ftf$.

We firstly review the construction of a single nontrivial regular $(243,22,1,2)$-PDS in the additive group of $\Ftf$. This PDS was originally constructed from the perfect ternary Golay code \cite{BLS}; we shall use the following alternative description involving a group of collineations of projective space having exactly two point-orbits \cite[Section 10]{CK}.
The $\frac{3^5-1}{3-1}=121$ points of the projective space $\PG(4,3)$ are the $1$-dimensional subspaces of the vector space $\Ftf$ over $\Ft$. Each such point has the form $\lan x \ran$ for some nonzero $x \in \Ftf$, and corresponds to the vectors $x$ and $2x$ of~$\Ftf$.
The general linear group $\GL(5,3)$ is the group of $5 \times 5$ invertible matrices over~$\Ft$, and its center is $Z=\{I, 2I\}$ where $I$ is the $5 \times 5$ identity matrix.
The projective linear group $\PGL(5,3)$ is the quotient group $\GL(5,3)/Z$.
The action of an element $A \in \PGL(5,3)$ on a point $\lan x \ran \in \PG(4,3)$ is given by
$$
A: \lan x \ran \mapsto \lan xA \ran,
$$
where $xA$ is the usual vector-matrix product, and this action is transitive on the points of $\PG(4,3)$ \cite[p. 57]{DM}.
Now $\PGL(5,3)$ contains a subgroup of order $7920$ which is a representation of the Mathieu group $M_{11}$. The group $M_{11}$ has exactly two point-orbits on $\PG(4,3)$: one of size $11$ and the other of size $110$ \cite[Example RT6]{CK}.
The $22$ vectors of $\Ftf$ corresponding to the point-orbit of size 11 form a nontrivial regular $(243,22,1,2)$-PDS in the additive group of $\Ftf$ \cite[Theorem 3.2 and Figure 2b]{CK}.

Define the elements of $\PGL(5,3)$:
\[
X=\begin{bmatrix}  0 &2 &1 &0 &0\\  2 &1 &1 &2 &2\\  0 &1 &1 &2 &2\\
  1 &0 &2 &2 &1\\  1 &2 &2 &2 &0\end{bmatrix} \quad \mbox{and} \quad
Y=\begin{bmatrix}  0 &0 &2 &0 &2\\  1 &1 &2 &2 &0\\
  2 &2 &2 &2 &2\\
  1 &2 &1 &1 &0\\  2 &2 &0 &2 &1
\end{bmatrix},
\]
which satisfy $X^2 = Y^4 = (XY)^{11} = I$.
The group $M_{11}$ may be represented explicitly \cite{ATLAS} as
\[
M_{11} = \lan X,Y \ran.
\]
The software package \emph{Magma} gives the
point-orbit of size $11$ under the action of $M_{11}$ on $\PG(4,3)$ as
\begin{align}
O_1 = \{
 &\lan(1,0,0,0,0)\ran,\lan(1,1,0,0,2)\ran,\lan(2,2,1,0,1)\ran,\lan(1,0,2,1,0)\ran,\lan(0,0,2,1,2)\ran,\lan(0,1,2,0,0)\ran, 	\nonumber \\
 &\lan(0,0,1,0,1)\ran,\lan(2,0,0,2,1)\ran,\lan(2,2,1,2,0)\ran,\lan(0,1,0,1,2)\ran,\lan(0,2,0,2,0)\ran\},			\label{eqn-O1explicit}
\end{align}
and the corresponding nontrivial regular $(243,22,1,2)$-PDS is
\begin{align}
B_1 =\{ & x \mid \lan x \ran \in O_1\} \cup \{2x \mid \lan x \ran \in O_1 \} \nonumber \\
    =\{
  &(1,0,0,0,0),(1,1,0,0,2),(2,2,1,0,1),(1,0,2,1,0),(0,0,2,1,2),(0,1,2,0,0),\nonumber \\
  &(0,0,1,0,1),(2,0,0,2,1),(2,2,1,2,0),(0,1,0,1,2),(0,2,0,2,0),(2,0,0,0,0),\nonumber \\
  &(2,2,0,0,1),(1,1,2,0,2),(2,0,1,2,0),(0,0,1,2,1),(0,2,1,0,0),(0,0,2,0,2),\nonumber \\
  &(1,0,0,1,2),(1,1,2,1,0),(0,2,0,2,1),(0,1,0,1,0)\} \label{eqn-B1defn}
\end{align}
in the additive group of $\Ftf$.

It is convenient to write
\[
W = XY =
\begin{bmatrix}
  1 &1 &0 &0 &2\\
  0 &2 &1 &1 &2\\
  0 &2 &0 &1 &1\\
  2 &1 &2 &2 &1\\
  2 &1 &0 &1 &0
\end{bmatrix},
\]
giving the alternative representation
\[
M_{11} = \lan W, Y \ran.
\]
Now the cyclic group $\lan W \ran$ is an order 11 subgroup of $M_{11}$. The orbit of $\lan (1,0,0,0,0) \ran$ under the action of $\lan W \ran$ has size 1 or 11; since $\lan W \ran$ does not fix the point $\lan (1,0,0,0,0) \ran$, this orbit is the whole of $O_1$:
\begin{equation}\label{eqn-O1orbit}
O_1= \{\lan(1,0,0,0,0) A\ran \mid A \in \lan W \ran \}.
\end{equation}
Recall that the group $M_{11}$ has exactly two point-orbits on $\PG(4,3)$: one of size $11$ (the set $O_1$), and the other of size $110$.
We will show that the action of the cyclic subgroup $\lan W \ran$ of $M_{11}$ on the points of $\PG(4,3)$ breaks the point-orbit of size 110 (under the action of $M_{11}$) into $10$ point-orbits of size $11$, each of which also corresponds to a nontrivial regular $(243,22,1,2)$-PDS in the additive group of $\Ftf$. This will give the partition of the nonzero elements of $\Ftf$ into 11 subsets required under Lemma~\ref{lem-PDS}.

The centralizer of $W$ in $\PGL(5,3)$ is the group $C(W) = \{B \in \PGL(5,3) : BW=WB\}$. \emph{Magma} gives $C(W)$ to be a cyclic group of order 121, one of whose generators is
$$
S=
\begin{bmatrix}
  1 &2 &2 &1 &2\\
  1 &2 &0 &1 &2\\
  0 &2 &2 &1 &2\\
  0 &0 &0 &0 &2\\
  1 &1 &1 &0 &1
\end{bmatrix},
$$
which satisfies $W = S^{11}$.
Define subsets $O_2, O_3, \dots, O_{11}$ of $\PG(4,3)$ by
\begin{equation}\label{eqn-OjxS}
O_j = \{ \lan xS^{j-1} \ran \mid \lan x \ran \in O_1\} \quad \mbox{for $2 \le j \le 11$}.
\end{equation}
Then for $1 \le j \le 11$ we find from \eqref{eqn-O1orbit} that
\begin{align*}
O_j &= \{ \lan (1,0,0,0,0)AS^{j-1} \ran \mid A \in \lan W \ran\}  \\
    &= \{ \lan (1,0,0,0,0)S^{j-1}A \ran \mid A \in \lan W \ran\}
\end{align*}
because $S \in C(W)$, and therefore the subset $O_j$ is the size $11$ orbit of the point $\lan (1,0,0,0,0)S^{j-1} \ran$ under the action of $\lan W \ran$.
Furthermore, using $W = S^{11}$ we may write
\begin{equation}\label{eqn-Ojkj}
O_j = \{ \lan (1,0,0,0,0)S^{11i+j-1} \ran \mid 0 \le i \le 10 \} \quad \mbox{for $1 \le j \le 11$},
\end{equation}
so that
\begin{equation}\label{eqn-unionOj}
\bigcup_{j=1}^{11} O_j=\{ \lan (1,0,0,0,0)S^{\ell} \ran \mid 0 \le \ell \le 120 \}.
\end{equation}
We claim that the subsets $O_1,O_2,\ldots,O_{11}$ form a partition of the 121 points of $\PG(4,3)$. Suppose, for a contradiction, that there is an integer $n$ satisfying $1 \le n \le 120$ such that
\begin{equation}\label{eqn-1Sm}
\lan (1,0,0,0,0)S^n \ran = \lan (1,0,0,0,0) \ran.
\end{equation}
Since $\lan S \ran = C(W)$ has order $121$, the matrix $S^n$ has order $11$ or $121$.
But $S^n$ cannnot have order 121, otherwise $S$ would fix the point $\lan (1,0,0,0,0)\ran$ and then from \eqref{eqn-unionOj} we would have $\bigcup_{j=1}^{11} O_j=\{\lan (1,0,0,0,0) \ran\}$, contradicting~\eqref{eqn-O1explicit}.
Therefore $S^n$ has order 11, so $S^n = S^{11i}$ for some $i$ satisfying $1 \le i \le 10$.
But from \eqref{eqn-Ojkj} the 11 points
$\{\lan(1,0,0,0,0)S^{11i}\ran \mid 0 \le i \le 10\}$ comprise the orbit $O_1$, and from \eqref{eqn-O1explicit} these 11 points are all distinct. This contradicts~\eqref{eqn-1Sm} and establishes the claim.

Finally, define subsets $B_2, B_3, \dots, B_{11}$ of the nonzero elements of $\Ftf$ by setting
\[
B_j = \{ x \mid \lan x \ran \in O_j\} \cup \{2x \mid \lan x \ran \in O_j \}
\quad \mbox{for $2 \le j \le 11$}.
\]
The subsets $B_1, B_2, \dots, B_{11}$ partition the 242 nonzero elements of $\Ftf$, and from \eqref{eqn-B1defn} and \eqref{eqn-OjxS} we have
\begin{equation}\label{eqn-Bjdefn}
B_j = \{ xS^{j-1} \mid x \in B_1\} \quad \mbox{for $1 \le j \le 11$}.
\end{equation}
Moreover, $B_1$ is a nontrivial regular $(243,22,1,2)$-PDS in the additive group of $\Ftf$, so from the definition \eqref{defn-PDS} the multiset $\{ x - y \mid x, y \in B_1\}$ contains the element 0 exactly $22$ times, each element of $B_1$ exactly once, and each other element of $\Ftf$ exactly twice.
Since $S$ is invertible, it follows from \eqref{eqn-Bjdefn} that each $B_j$ is also a nontrivial regular $(243,22,1,2)$-PDS in the additive group of~$\Ftf$ for $2 \le j \le 11$. By Lemma~\ref{lem-PDS}, $\{B_1,B_2,\ldots,B_{11}\}$ is therefore a $(243,11,22,20)$ near-complete SEDF in the additive group of~$\Ftf$.

Explicitly, we have
\begin{align*}
B_2=
\{&(1,2,2,1,2),(1,0,1,2,0),(2,2,1,2,2),(1,0,0,0,2),(2,0,0,2,2),(1,0,1,0,0),\\
&(1,0,0,1,0),(0,2,2,2,0),(1,1,0,2,2),(0,1,2,1,0),(2,1,0,2,2),(2,1,1,2,1),\\
&(2,0,2,1,0),(1,1,2,1,1),(2,0,0,0,1),(1,0,0,1,1),(2,0,2,0,0),(2,0,0,2,0),\\
&(0,1,1,1,0),(2,2,0,1,1),(0,2,1,2,0),(1,2,0,1,1)\},\\
B_3=
\{&(2,0,2,2,2),(1,1,1,2,2),(0,0,2,2,1),(0,1,1,1,1),(1,0,0,2,1),(1,1,1,2,1),\\
&(1,2,2,1,1),(2,2,1,1,0),(1,0,1,2,1),(1,0,1,0,2),(2,2,0,0,0),(1,0,1,1,1),\\
&(2,2,2,1,1),(0,0,1,1,2),(0,2,2,2,2),(2,0,0,1,2),(2,2,2,1,2),(2,1,1,2,2),\\
&(1,1,2,2,0),(2,0,2,1,2),(2,0,2,0,1),(1,1,0,0,0)\},\\
B_4=
\{&(1,1,1,1,2),(1,2,0,0,0),(1,2,2,2,0),(2,2,0,2,1),(2,0,0,1,1),(0,1,2,0,2),\\
&(1,2,1,2,1),(1,1,0,2,0),(2,2,2,2,0),(0,0,0,2,0),(1,2,1,1,2),(2,2,2,2,1),\\
&(2,1,0,0,0),(2,1,1,1,0),(1,1,0,1,2),(1,0,0,2,2),(0,2,1,0,1),(2,1,2,1,2),\\
&(2,2,0,1,0),(1,1,1,1,0),(0,0,0,1,0),(2,1,2,2,1)\},\\
B_5=
\{&(1,2,0,0,1),(0,0,2,0,0),(0,1,0,2,2),(2,0,2,1,1),(0,2,2,2,1),(0,2,0,0,2),\\
&(1,0,2,1,1),(2,1,2,2,2),(1,0,2,0,1),(0,0,0,0,1),(2,1,0,1,0),(2,1,0,0,2),\\
&(0,0,1,0,0),(0,2,0,1,1),(1,0,1,2,2),(0,1,1,1,2),(0,1,0,0,1),(2,0,1,2,2),\\
&(1,2,1,1,1),(2,0,1,0,2),(0,0,0,0,2),(1,2,0,2,0)\},\\
B_6=
\{&(1,1,0,0,1),(0,1,1,2,1),(0,1,2,1,2),(0,0,0,1,2),(0,0,2,1,1),(1,0,2,2,0),\\
&(2,1,1,0,0),(2,0,1,2,1),(2,1,1,0,1),(1,1,1,0,1),(0,0,1,0,2),(2,2,0,0,2),\\
&(0,2,2,1,2),(0,2,1,2,1),(0,0,0,2,1),(0,0,1,2,2),(2,0,1,1,0),(1,2,2,0,0),\\
&(1,0,2,1,2),(1,2,2,0,2),(2,2,2,0,2),(0,0,2,0,1)\},\\
B_7=
\{&(0,2,0,2,2),(2,2,0,2,0),(0,2,0,0,1),(2,2,2,0,1),(1,2,2,2,1),(1,0,0,0,1),\\
&(0,2,0,1,2),(0,1,1,0,2),(1,0,1,1,0),(0,1,2,0,1),(2,1,1,1,1),(0,1,0,1,1),\\
&(1,1,0,1,0),(0,1,0,0,2),(1,1,1,0,2),(2,1,1,1,2),(2,0,0,0,2),(0,1,0,2,1),\\
&(0,2,2,0,1),(2,0,2,2,0),(0,2,1,0,2),(1,2,2,2,2)\},\\
B_8=
\{&(1,0,2,2,1),(1,2,1,1,0),(0,2,1,2,2),(2,1,0,0,1),(1,2,1,2,0),(2,0,0,1,0),\\
&(1,0,2,2,2),(0,0,1,2,0),(1,1,1,2,0),(2,1,2,0,1),(1,0,1,1,2),(2,0,1,1,2),\\
&(2,1,2,2,0),(0,1,2,1,1),(1,2,0,0,2),(2,1,2,1,0),(1,0,0,2,0),(2,0,1,1,1),\\
&(0,0,2,1,0),(2,2,2,1,0),(1,2,1,0,2),(2,0,2,2,1)\},\\
B_9=
\{&(2,1,1,0,2),(0,2,1,1,1),(1,2,1,0,0),(1,1,2,0,1),(0,2,1,1,0),(2,1,1,2,0),\\
&(0,2,2,0,0),(0,2,2,1,0),(2,0,1,0,1),(1,2,0,2,2),(0,0,0,2,2),(1,2,2,0,1),\\
&(0,1,2,2,2),(2,1,2,0,0),(2,2,1,0,2),(0,1,2,2,0),(1,2,2,1,0),(0,1,1,0,0),\\
&(0,1,1,2,0),(1,0,2,0,2),(2,1,0,1,1),(0,0,0,1,1)\},\\
B_{10}=
\{&(2,1,2,1,1),(0,1,0,0,0),(0,2,1,1,2),(0,0,1,1,0),(2,0,2,0,2),(0,2,0,1,0),\\
&(2,2,1,1,2),(2,2,1,1,1),(0,1,1,0,1),(2,2,1,0,0),(2,2,2,0,0),(1,2,1,2,2),\\
&(0,2,0,0,0),(0,1,2,2,1),(0,0,2,2,0),(1,0,1,0,1),(0,1,0,2,0),(1,1,2,2,1),\\
&(1,1,2,2,2),(0,2,2,0,2),(1,1,2,0,0),(1,1,1,0,0)\},\\
B_{11}=
\{&(1,2,0,2,1),(1,2,0,1,2),(1,2,1,0,1),(0,2,2,1,1),(1,1,1,1,1),(2,1,0,2,0),\\
&(0,0,2,2,2),(2,2,1,2,1),(2,2,0,2,2),(1,1,0,2,1),(1,0,2,0,0),(2,1,0,1,2),\\
&(2,1,0,2,1),(2,1,2,0,2),(0,1,1,2,2),(2,2,2,2,2),(1,2,0,1,0),(0,0,1,1,1),\\
&(1,1,2,1,2),(1,1,0,1,1),(2,2,0,1,2),(2,0,1,0,0)\}.
\end{align*}

\section{An exponent bound and its application}
\label{sec-exp-bound}

In this section, we present an exponent bound on a group $G$ containing a $(v,m,k,\la)$-SEDF, and use it to prove nonexistence results for the case $m=2$.

Let $G = H \times L$ be an abelian group. Each element of $G$ can be expressed uniquely as $h \ell$ for $h \in H$ and $\ell \in L$, and the \emph{natural projection} $\rho$ from $G$ to $H$ is defined by $\rho(h \ell) = h$. Each $\wti{\chi} \in \wh{H}$ induces a \emph{lifting character} $\chi \in \wh{G}$ satisfying $\chi(g)=\wti{\chi}(\rho(g))$ for every $g \in G$. From now on, we shall use $G_p$ to denote the Sylow $p$-subgroup of the group $G$, where $p$ is a prime. 
For a positive integer $n$, we use $\ze_n$ to denote the primitive $n$-th complex root of unity $e^{2\pi i/n}$.

A prime $p$ is a \emph{primitive root} modulo $n$ if 
$p$ is a generator of the multiplicative group of integers modulo $n$.
A prime $p$ is \emph{self-conjugate} modulo $n$ if there is an integer $j$ for which $p^j \equiv -1 \pmod{n_p}$, where $n_p$ is the largest divisor of $n$ that is not divisible by~$p$. 
If a prime $p$ is a primitive root modulo $n$, then $p$ is self-conjugate modulo~$n$.
For $X \in \Z[\ze_n]$, we use $(X)$ to denote the principal idea generated by $X$ in $\Z[\ze_n]$. We begin with a preparatory lemma.

\begin{lemma}\label{lem-XX'}
Let $p$ and $q$ be primes, let $q$ be a primitive root modulo $p^e$, and let $q^{f} \mid\mid u$ for some positive integer~$f$.
Suppose that $X, X^{\pr} \in \Z[\ze_{p^e}]$ satisfy $X\ol{X^{\pr}}=u$.
Then either $X \equiv 0 \pmod{q^{\lc f/2 \rc}}$ or $X^{\pr} \equiv 0 \pmod{q^{\lc f/2 \rc}}$. Furthermore, if $X=X^{\pr}$, then $f$ is even.
\end{lemma}
\begin{proof}
Since $X\ol{X^{\pr}}=u$ and $q^f \mid\mid u$, we have $X\ol{X^{\pr}} \equiv 0 \pmod{q^f}$. Now $q$ is a primitive root modulo $p^e$, so $(q)$ is a prime ideal in $\Z[\ze_{p^e}]$ \cite[Chapter 13, Theorem 2]{IR}, which we denote by $\cQ$. Hence
$$
X\ol{X^{\pr}} \equiv 0 \pmod{\cQ^f}
$$
and so
$$
\cQ^f \mid (X)(\ol{X^{\pr}}).
$$
Therefore either $\cQ^{\lc f/2 \rc} \mid (X)$ or $\cQ^{\lc f/2 \rc} \mid (X^{\pr})$, and so either $X \equiv 0 \pmod{q^{\lc f/2 \rc}}$ or $X^{\pr} \equiv 0 \pmod{q^{\lc f/2 \rc}}$.

Now suppose $X=X^{\pr}$, so that $\cQ^{\lc f/2 \rc} \mid (X)$. Since $q$ is a primitive root modulo $p^e$, we have that $q$ is self-conjugate modulo~$p^e$. This implies $\cQ$ is invariant under complex conjugation \cite[Chapter VI, Corollary 15.5]{BJL}, so that $\cQ^{\lc f/2 \rc} \mid (\ol{X})$ and therefore $\cQ^{2\lc f/2 \rc} \mid (X)(\ol{X})$. But $q^f \mid\mid u$, so $\cQ^f \mid\mid (X)(\ol{X})$. Therefore $f$ is even.
\end{proof}

We now prove the following exponent bound.

\begin{theorem}\label{thm-exp}
Suppose there exists a $(v,m,k,\la)$-SEDF in a group $G$.
Let $p$ and $q$ be primes such that $p^d \mid\mid v$ and $q^f \mid\mid \la$ for some positive integers $d$ and $f$, and suppose that $q$ is a primitive root modulo $p^d$. Let $G_p$ be the Sylow $p$-subgroup of $G$. Then
$$
\exp(G_p) \le v/q^{\lc f/2\rc}.
$$
\end{theorem}
\begin{proof}
Let $\{D_1,D_2,\ldots,D_m\}$ be the SEDF, and let $D = \bigcup_{i=1}^m D_i$.
Let $\exp(G_p)=p^e$, let $H$ be a cyclic $p$-subgroup of $G$ of order~$p^e$ occurring as a direct factor of~$G$, and let $\rho$ be the natural projection from $G$ to $H$. Let $\wti{\chi}$ be a generator of $\wh{H}$, and let $\chi$ be the associated lifting character on~$G$. Then
\begin{equation}\label{eqn-minusla}
\wti{\chi}(\rho(D_1))\ol{\wti{\chi}(\rho(D-D_1))} =
\chi(D_1)\ol{\chi(D-D_1)} =-\la
\end{equation}
by (\ref{eqn-chardef}).
Now $q$ is a primitive root modulo $p^d$, so $q$ is also a primitive root modulo $p^e$ \cite[Chapter~4, Lemma~3]{IR}. Apply Lemma~\ref{lem-XX'} with $X = \wti{\chi}(\rho(D_1))$ and $X' = \wti{\chi}(\rho(D-D_1))$ and $u = -\lambda$ to show that there is a subset $D'$ of $G$ (either $D_1$ or $D\sm D_1$) for which
$\wti{\chi}(\rho(D')) \equiv 0 \pmod{q^{\lc f/2\rc}}$.
Write 
\begin{equation}\label{eqn-rhoD'di}
\wti{\chi}(\rho(D'))=\sum_{i=0}^{p^{e-1}-1} d_i \ze_{p^e}^i, 
\end{equation}
where $d_i=\sum_{j=0}^{p-1} c_{i,j} \ze_p^j$ and each $c_{i,j} \in \Z$. We have shown that
$q^{\lc f/2\rc} \mid \sum_{i=0}^{p^{e-1}-1} d_i \ze_{p^e}^i$. 
Since $\{1,\ze_{p^e},\ze_{p^e}^2,\ldots,\ze_{p^e}^{p^{e-1}-1}\}$ is a linearly independent set over $\Q[\ze_p]$, this implies that
\[
q^{\lc f/2\rc} \mid d_i \quad \mbox{for each $i$}.
\]
Note that $d_i = \sum_{j=0}^{p-1} c_{i,j} \ze_p^j =\sum_{j=0}^{p-2} (c_{i,j}-c_{i,p-1})\ze_p^j$. 
Because $\{1,\ze_p,\ze_p^2,\ldots,\ze_p^{p-2}\}$ is an integral basis of $\Z[\ze_{p}]$, we then obtain
\begin{equation}\label{eqn-qf2cij}
q^{\lc f/2\rc} \mid (c_{i,j}-c_{i,p-1}) \quad \mbox{for each $i$ and $j$}.
\end{equation}
Since $\rho$ is the natural projection from $G$ to $H$, and $|H|=p^e$, we have $0 \le c_{i,j} \le \frac{v}{p^e}$ for each $i$ and $j$, and therefore
$-\frac{v}{p^e} \le c_{i,j}-c_{i,p-1} \le \frac{v}{p^e}$ for each $i$ and $j$.
Furthermore, from \eqref{eqn-minusla} and~\eqref{eqn-rhoD'di}, $d_i' \ne 0$ for some $i'$ and therefore $c_{i',j'}-c_{i',p-1} \ne 0$ for some $j'$.
It follows from \eqref{eqn-qf2cij} that
$q^{\lc f/2\rc} \le v/p^e$, or equivalently $p^e \le v/q^{\lc f/2\rc}$.
\end{proof}

Very few nonexistence results for a nontrivial $(v,m,k,\la)$-SEDF with $m=2$ are known. We now illustrate the use of Theorem~\ref{thm-exp} by ruling out several families of such parameter sets. When $m=2$ and $k$ is prime, the existence question is already answered: we must have $\la=1$ \cite[Lemma 3.4]{HP}, and then by Proposition~\ref{prop-smallla} the parameters have the form $(k^2+1,2,k,1)$. We therefore consider $m=2$ and $k = p_1 p_2$ in Theorem~\ref{thm-m2twoprimes} below, where $p_1, p_2$ are distinct primes and $p_1 < p_2$. The case $\la=1$ is dealt with in Proposition~\ref{prop-smallla}, and the cases $\la \ge p_1p_2$ are ruled out by Proposition~\ref{prop-nonexistence}~(6).
In view of the counting relation $p_1^2p_2^2 = \la (v-1)$ given by \eqref{eqn-counting}, the remaining cases are $\la \in \{p_1,p_2,p_1^2\}$.

\begin{theorem}\label{thm-m2twoprimes}
Let $p_1$ and $p_2$ be distinct primes with $p_1<p_2$.
\begin{enumerate}[(1)]
\item Let $p$ be a prime such that $p^d \mid\mid p_1p_2^2+1$ for some positive integer~$d$. If $p_1$ is a primitive root modulo $p^d$ and $p_2^2+1 \le p$, then a $(p_1p_2^2+1,2,p_1p_2,p_1)$-SEDF does not exist.
\item Let $p$ be a prime such that $p^d \mid\mid p_1^2p_2+1$ for some positive integer~$d$. If $p_2$ is a primitive root modulo $p^d$ and $p_1^2+1 \le p$, then a $(p_1^2p_2+1,2,p_1p_2,p_2)$-SEDF does not exist.
\item Let $p$ be a prime such that $p^d \mid\mid p_2^2+1$ for some positive integer~$d$. If $p_1$ is a primitive root modulo $p^d$ and $p_2^2+1<p_1p$, or if $p_1$ is a primitive root modulo $p$ and $p_2^2+1=p_1p$, then a $(p_2^2+1,2,p_1p_2,p_1^2)$-SEDF does not exist.
\end{enumerate}
\end{theorem}
\begin{proof}
Parts (1), (2) and the first part of (3) are each direct applications of Theorem~\ref{thm-exp}, whereas the second part of (3) requires additional arguments; we give the proof for both parts of~(3).
Suppose, for a contradiction, that $\{D_1,D_2\}$ is a $(p_2^2+1,2,p_1p_2,p_1^2)$-SEDF in a group $G$ of order~$p_2^2+1$.

If $p_1$ is a primitive root modulo $p^d$ and $p_2^2+1 < p_1p$, then by Theorem~\ref{thm-exp} we have $\exp(G_p) \le \frac{p_2^2+1}{p_1} < p$. This contradicts that $p$ is a prime divisor of $|G| = p_2^2+1$.

If $p_1$ is a primitive root modulo $p$ and $p_2^2+1=p_1p$, then $G = \Z_{p_1} \times \Z_{p}$.
Let $\rho$ be the natural projection from $G$ to $\Z_{p}$.
Let $\wti{\chi}$ be a generator of $\wh{\Z_p}$.
Then
\[
\wti{\chi}(\rho(D_1))\ol{\wti{\chi}(\rho(D_2))} = -p_1^2
\]
by \eqref{eqn-chardef}, so by Lemma~\ref{lem-XX'} we may choose $D'$ to be one of $D_1$ and $D_2$ so that
$\wti{\chi}(\rho(D')) \equiv 0 \pmod{p_1}$.
Since $\wti{\chi}$ is a generator of $\wh{\Z_p}$, there is a generator $h$ of $\Z_p$ for which $\wti{\chi}(h)=\ze_p$.
Write $\rho(D')=\sum_{i=0}^{p-1}d_ih^i$, where $0 \le d_i \le p_1$ for each $i$, and then
$$
\wti{\chi}(\rho(D'))=\sum_{i=0}^{p-1}d_i\ze_{p}^i = \sum_{i=0}^{p-2}(d_i-d_{p-1})\ze_{p}^i.
$$
Since $\wti{\chi}(\rho(D')) \equiv 0 \pmod{p_1}$, we have
$p_1 \mid d_i-d_{p-1}$ for each $i$. Using $0 \le d_i \le p_1$ for each $i$, we distinguish two cases:
\begin{description}
\item[Case 1:] $d_i \in \{0, p_1\}$ for each $i$ satisfying $0 \le i \le p-1$.
This gives $\rho(D') = p_1 \sum_{i \in I} h^i$ for some subset $I$ of $\{0,1,\dots,p-1\}$, which implies that $D'$ is a union of cosets of $\Z_{p_1}$. But then for a character $\chi \in \wh{G}$ which is nonprincipal on $\Z_{p_1}$ we have $\chi(D') = 0$, contradicting~\eqref{eqn-chardef} because $D' = D_1$ or~$D_2$.

\item[Case 2:] $d_0=d_1=\dots=d_{p-1}$.
Then $p$ divides $|D'| = p_1 p_2$, so either $p=p_1$ or $p=p_2$. Both of these contradict the given conditions on $p, p_1, p_2$.
\end{description}
\end{proof}

\begin{remark}
For example, Theorem~\ref{thm-m2twoprimes} rules out the existence of a $(v,m,k,\la)$-SEDF for
\begin{align*}
(v,m,k,\la) \in \{ &(19,2,6,2), (26,2,10,4), (46,2,15,5), (118,2,39,13), (122,2,22,4), \\
		   &(154,2,51,17), (172,2,57,19) \}.
\end{align*}
Theorem~\ref{thm-exp} rules out further parameter sets not excluded by Theorem~\ref{thm-m2twoprimes} (for which $k$ is not the product of two distinct primes), including
\begin{align*}
(v,m,k,\la) \in \{ &(37,2,12,4), (101,2,20,4), (101,2,30,9), (101,2,40,16), (122,2,44,16), \\
		   &(127,2,42,14), (129,2,48,18), (163,2,18,2), (163,2,36,8), (163,2,54,18),  \\
                   &(163,2,72,32), (177,2,44,11), (181,2,60,20), (197,2,28,4), (197,2,42,9), \\
		   &(197,2,56,16), (197,2,70,25), (197,2,84,36) \}.
\end{align*}
All known examples of a nontrivial $(v,2,k,\la)$-SEDF have $v$ a prime, except for those specified in Proposition~$\ref{prop-known}~(1)$ and~$(2)$.
The only cases for a $(v,2,k,\la)$-SEDF with $v \le 50$ that remain open are
$$
(v,m,k,\la) \in \{ (28,2,9,3), (33,2,8,2), (49,2,12,3), (50,2,14,4), (50,2,21,9) \}.
$$
The parameter set $(45,2,22,11)$ does not appear on this list, even though the existence of an SEDF with these parameters in $\Z_3^2 \times \Z_5$ is not ruled out by Theorems~\ref{thm-exp} and \ref{thm-m2twoprimes}: by Theorem~\ref{thm-characterization}, the existence of an SEDF with these parameters would imply the existence of a nontrivial regular $(45,22,10,11)$-PDS, which is excluded by \cite[Corollary 6.3]{Ma94}.
\end{remark}

\section{SEDFs with $m>2$}
\label{sec-m>2}
Throughout this section, we suppose that $\{D_1,D_2,\dots,D_m\}$ is a nontrivial $(v,m,k,\la)$-SEDF in a group $G$ with $m>2$, and write $D = \bigcup_{i=1}^m D_i$. From \eqref{eqn-solsfre}, in order for $\lpc$ and $\lmc$ to be integers we must have
\begin{equation}\label{eqn-bcac}
\sqrt{1+\tfrac{4\la}{|\chi(D)|^2}}=\frac{\bc}{\ac} \quad \mbox{for each $\chi \in \wh{G}^N$}, \quad
\mbox{where $\ac, \bc \in \Z$ and $\bc>\ac>0$ and $\gcd(\ac,\bc)=1$}.  	
\end{equation}
Then \eqref{eqn-solsfre} becomes
\begin{equation}\label{eqn-solsfresim}
\lpc=\frac{m}{2}-\frac{\ac(m-2)}{2\bc}, \quad
\lmc=\frac{m}{2}+\frac{\ac(m-2)}{2\bc}
\quad \mbox{for each $\chi \in \wh{G}^N$}.
\end{equation}
It is shown in \cite[Lemma 3.3]{BJWZ} and \cite[Lemma 3.5]{MS} that $(\lpc,\lmc) \notin \{(0,m), (1,m-1), (\frac{m}{2},\frac{m}{2})\}$ for $m>2$; this is an immediate consequence of~\eqref{eqn-solsfresim}. Rewrite the expressions \eqref{eqn-sols} for $\alp$ and $\alm$ using \eqref{eqn-bcac}, and then substitute into \eqref{eqn-chiDj} to obtain
\begin{equation}\label{eqn-chiDjsim}
\{\chi(D_j) \mid 1 \le j \le m\} = \Big\{ \frac{\ac+\bc}{2\ac}\chi(D), \frac{\ac-\bc}{2\ac}\chi(D) \Big\}
\quad \mbox{for each $\chi \in \wh{G}^N$}.
\end{equation}
Rearrange \eqref{eqn-bcac} as
\[
|\chi(D)|^2=\frac{4\ac^2\la}{\bc^2-\ac^2} \quad \mbox{for each $\chi \in \wh{G}^N$},
\]
and then combine with \eqref{eqn-chiDjsim} to give
\begin{align}
& \Big\{\big(|\chi(D)|^2,|\chi(D_j)|^2\big) \mid 1 \le j \le m \Big\} \nonumber \\
& \hspace{5em} = \bigg\{
\bigg(\frac{4\ac^2\la}{\bc^2-\ac^2}, \, \frac{(\bc+\ac)\la}{\bc-\ac}\bigg),
\bigg(\frac{4\ac^2\la}{\bc^2-\ac^2}, \, \frac{(\bc-\ac)\la}{\bc+\ac}\bigg)
\bigg\}
\quad \mbox{for each $\chi \in \wh{G}^N$}.  \label{eqn-chiDjsqsim}
\end{align}

We now derive some divisibility conditions on the values of $\ac$ and $\bc$, which restrict the possible values of $|\chi(D)|^2$ and $|\chi(D_j)|^2$ via~(\ref{eqn-chiDjsqsim}).

\begin{lemma}\label{lem-abpair}
Let $\ac,\bc$ be defined as in \eqref{eqn-bcac} (with reference to the set $D = \bigcup_{i=1}^m D_i$ associated with a nontrivial $(v,m,k,\la)$-SEDF $\{D_1,D_2,\dots,D_m\}$ in a group $G$ with $m>2$).
Then
\begin{enumerate}[(1)]
\item $2\bc \mid \bc m - \ac(m-2)$, and $\bc \mid m-2$
\item $(\bc-\ac) \mid (\bc+\ac)\la$, and $(\bc+\ac) \mid (\bc-\ac)\la$
\item $(\bc^2-\ac^2) \mid 4\la$, and if $\bc+\ac$ is odd then $(\bc^2-\ac^2) \mid \la$.
\end{enumerate}
\end{lemma}
\begin{proof}
\mbox{}
\begin{enumerate}[(1)]
\item
Since $\lpc$ is an integer, by \eqref{eqn-solsfresim} we have $2\bc \mid \bc m - \ac(m-2)$. Therefore $\bc \mid \ac(m-2)$, and since $\gcd(\ac,\bc)=1$ we have $\bc \mid m-2$.

\item
$|\chi(D_j)|^2$ is an algebraic integer, and by \eqref{eqn-chiDjsqsim}
also takes both the rational values $\frac{(\bc+\ac)\la}{\bc-\ac}$ and $\frac{(\bc-\ac)\la}{\bc+\ac}$ as $j$ ranges over $\{1,2,\dots,m\}$. Therefore
$\frac{(\bc+\ac)\la}{\bc-\ac}$ and $\frac{(\bc-\ac)\la}{\bc+\ac}$ are both integers.

\item
$|\chi(D)|^2$ is an algebraic integer, and by \eqref{eqn-chiDjsqsim} is also the rational number $\frac{4\ac^2\la}{\bc^2-\ac^2}$. Therefore $\frac{4\ac^2\la}{\bc^2-\ac^2}$ is an integer, which implies $(\bc^2-\ac^2) \mid 4\la$.
If $\bc+\ac$ is odd, then $\gcd(\bc-\ac,\bc+\ac) = \gcd(2\bc,\bc+\ac) = \gcd(\bc,\bc+\ac)=1$ so that from part (2) we obtain $(\bc-\ac) \mid \la$ and $(\bc+\ac) \mid \la$ and therefore $(\bc^2-\ac^2) \mid \la$.

\end{enumerate}
\end{proof}

Using Lemma~\ref{lem-abpair}, we recover the result of Proposition~\ref{prop-nonexistence}~(1) as Corollary~\ref{cor-m34}, and obtain new restrictions for $m \in \{5,6\}$ as Corollary~\ref{cor-scvp}.

\begin{corollary}\label{cor-m34}
A nontrivial $(v,m,k,\la)$-SEDF does not exist for $m \in \{3,4\}$.
\end{corollary}

\begin{corollary}\label{cor-scvp}
Let $\ac,\bc$ be defined as in \eqref{eqn-bcac}.
\begin{enumerate}[(1)]
\item If there exists a nontrivial $(v,5,k,\la)$-SEDF in a group $G$, then $(\ac,\bc)=(1,3)$ and $2 \mid \la$
for each $\chi \in \wh{G}^N$.
\item If there exists a nontrivial $(v,6,k,\la)$-SEDF in a group $G$, then $(\ac,\bc)=(1,2)$ and $3 \mid \la$
for each $\chi \in \wh{G}^N$.
\end{enumerate}
\end{corollary}

Motivated by Corollary~\ref{cor-scvp}, we say that a nontrivial $(v,m,k,\la)$-SEDF with $m>2$ for which $(\ac,\bc)$ takes a constant value $(a,b)$ for all $\chi \in \wh{G}^N$ has the \emph{simple character value property} with respect to $(a,b)$.
In the following subsection we obtain restrictions on SEDFs having this property. In particular, for $m=5$ and for $m=6$ we obtain asymptotic nonexistence results for a family of SEDFs, each of which must have this property with respect to a fixed $(a,b)$ by Corollary~\ref{cor-scvp}.

\subsection{The simple character value property}
\label{subsec-scvp}
As above, suppose that $\{D_1,D_2,\dots,D_m\}$ is a nontrivial $(v,m,k,\la)$-SEDF in a group $G$ with $m>2$, and write $D = \bigcup_{i=1}^m D_i$. Suppose further that $\{D_1,D_2,\dots,D_m\}$ has the simple character value property with respect to $(a,b)$. Then by \eqref{eqn-chiDjsim}, we may partition $\wh{G}^N$ (with respect to $D_1$) into the disjoint union of the sets
\begin{align}
\wh{G}^{+} &=\{\chi \in \wh{G}^N \mid \chi(D_1)=\frac{a+b}{2a}\chi(D)\}, \label{eqn-G+defn} \\
\wh{G}^{-} &=\{\chi \in \wh{G}^N \mid \chi(D_1)=\frac{a-b}{2a}\chi(D)\}, \label{eqn-G-defn}
\end{align}
and from the definition \eqref{eqn-G0defn}, $\wh{G}$ is the disjoint union $\{\chi_0\}\cup\wh{G}^0\cup\wh{G}^+\cup\wh{G}^-$.
By \eqref{eqn-chiDlaiff} and \eqref{eqn-chiDjsqsim}, we then obtain the character values in Table~\ref{tab-charvalues}.

\begin{table}[htbp]
  \centering
  $\begin{array}{|c|c|c|c|}												 \hline
\hspace{5em}		& \hspace{5em}			& \hspace{5em}			& \hspace{5em}		\\ [-0.5ex]
\chi \in \wh{G} 	& |\chi(D)|^2			& \chi(D_1) 			& |\chi(D_1)|^2 	\\ [2ex]\hline
			&				&				&			\\ [-0.5ex]
\chi = \chi_0		& k^2m^2			& k				& k^2 			\\ [3ex]
\chi \in \wh{G}^0	& 0				&  				& \lambda 		\\ [4ex]
\chi \in \wh{G}^{+}	& \dfrac{4a^2\la}{b^2-a^2}	& \dfrac{a+b}{2a}\chi(D)	& \dfrac{(b+a)\la}{b-a}	\\ [4ex]
\chi \in \wh{G}^{-}	& \dfrac{4a^2\la}{b^2-a^2}	& \dfrac{a-b}{2a}\chi(D)	& \dfrac{(b-a)\la}{b+a}	\\ [4ex] \hline
  \end{array}$
  \caption{Character sums for an SEDF with $m>2$, having the simple character value property with respect to $(a,b)$}
  \label{tab-charvalues}
\end{table}

We now determine the size of the sets $\wh{G}^0$, $\wh{G}^+$, $\wh{G}^-$.

\begin{theorem}\label{thm-simchar}
Suppose $\{D_1,D_2,\ldots,D_m\}$ is a nontrivial $(v,m,k,\la)$-SEDF in a group $G$ with $m>2$, having the simple character value property with respect to $(a,b)$. Then the sizes of the sets $\wh{G}^0$, $\wh{G}^+$, $\wh{G}^-$ (defined as in \eqref{eqn-G0defn}, \eqref{eqn-G+defn}, \eqref{eqn-G-defn} with reference to the sets $D_1$ and $D=\bigcup_{i=1}^m D_i$) are
\begin{align}
|\wh{G}^0| &=(v-1)\left(1-\frac{(b^2-a^2)(v-km)m}{4a^2k(m-1)}\right), \label{eqn-whG0} \\
|\wh{G}^+| &=\frac{(v-1)(v-km)(b^2-a^2)((b-a)m+2a)}{8a^2bk(m-1)}, \nonumber \\
|\wh{G}^-| &=\frac{(v-1)(v-km)(b^2-a^2)((b+a)m-2a)}{8a^2bk(m-1)}, \nonumber
\end{align}
and each of $|\wh{G}^0|, |\wh{G}^+|, |\wh{G}^-|$ is a non-negative integer and $|\wh{G}^+|+|\wh{G}^-| > 0$.
\end{theorem}
\begin{proof}
Each of $|\wh{G}^0|, |\wh{G}^+|, |\wh{G}^-|$ is a non-negative integer by definition,  and $|\wh{G}^+|+|\wh{G}^-| = |\wh{G}^N| > 0$ by Lemma~\ref{lem-GNnonempty}.
Write $DD^{(-1)}=\sum_{g \in G}c_gg \in \Z[G]$. From Proposition~\ref{prop-fourier},
\[
c_1=\frac{1}{v}\sum_{\chi \in \wh{G}} |\chi(D)|^2.
\]
The left side $c_1 = |D| = km$ is the coefficient of the identity in the expression $DD^{(-1)}$, and the right side can be evaluated using Table~\ref{tab-charvalues} to give
\[
km=\frac{1}{v}\left(k^2m^2+\big(v-1-|\wh{G}^0|\big)\frac{4a^2\la}{b^2-a^2}\right).
\]
Substitute for $\la$ from the counting relation \eqref{eqn-counting} to obtain the required expression for~$|\wh{G}^0|$.

Similarly, write $D_1D_1^{(-1)}=\sum_{g \in G}d_gg \in \Z[G]$ and use Proposition~\ref{prop-fourier} and Table~\ref{tab-charvalues} to give
\[
k =\frac{1}{v}\Big(k^2+|\wh{G}^0|\,\la+|\wh{G}^+|\,\frac{(b+a)\la}{b-a}+|\wh{G}^-|\,\frac{(b-a)\la}{b+a}\Big).
\]
We now obtain the required expressions for $|\wh{G}^+|$ and $|\wh{G}^-|$ using \eqref{eqn-whG0} and the counting condition $|\wh{G}^0|+|\wh{G}^+|+|\wh{G}^-|=v-1$.
\end{proof}

We obtain the following asymptotic nonexistence result from Theorem~\ref{thm-simchar}.

\begin{theorem}\label{thm-asymp-scvp}
Let $m$, $\la$, $a$, $b$ be fixed positive integers, where $m>2$ and $b > a$ and $\gcd(a,b)=1$. Then for all sufficiently large~$k$, there does not exist a nontrivial $(v,m,k,\la)$-SEDF having the simple character value property with respect to~$(a,b)$.
\end{theorem}
\begin{proof}
Apply the condition $|\wh{G}^0| \ge 0$ to \eqref{eqn-whG0}, and rearrange to give the inequality
$$
\frac{v}{k} \le m + \frac{4a^2(m-1)}{m(b^2-a^2)}.
$$
Since $m$ and $\la$ are fixed, the counting relation \eqref{eqn-counting} shows that $v$ grows like $k^2$ as $k$ increases. Therefore for all sufficiently large $k$, the inequality in $v/k$ does not hold.
\end{proof}

As a consequence of Corollary~\ref{cor-scvp} and Theorem~\ref{thm-asymp-scvp}, we obtain the following asymptotic nonexistence result for $m \in \{5,6\}$.

\begin{corollary}\label{cor-asymp-m56}
Let $\la$ be a fixed positive integer.
Then for all sufficiently large~$k$, there does not exist a nontrivial $(v,5,k,\la)$-SEDF and there does not exist a nontrivial $(v,6,k,\la)$-SEDF.
\end{corollary}

We can obtain results similar to Corollary~\ref{cor-asymp-m56} for values of $m$ greater than~$6$. For example, suppose there exists a nontrivial $(v,7,k,\la)$-SEDF. From Lemma~\ref{lem-abpair} we find that $\la \bmod 12 \in \{0,4,6,8\}$, and that the SEDF has the simple character value property with respect to $(1,5)$ if $\la \bmod 12 = 6$ and with respect to $(3,5)$ if $\la \bmod 12 \in \{4,8\}$. Therefore for fixed $\la$ for which $\la \bmod 12 \ne 0$, for all sufficiently large $k$ there does not exist a nontrivial $(v,7,k,\la)$-SEDF. Likewise, for fixed $\la$ for which $\la \bmod 10 \ne 0$, for all sufficiently large $k$ there does not exist a nontrivial $(v,8,k,\la)$-SEDF.

We derive further divisibility conditions on the SEDF parameters in Theorem~\ref{thm-char-scvp} below. We first require two number-theoretic lemmas.

\begin{lemma}{\rm \cite[Lemma 2.3]{Hira}}\label{lem-rootofunity}
Let $p$ be a prime and let $e$ be a positive integer. Let $\sig=\sum_{i=0}^{p^e-1} c_i\ze_{p^e}^{i}$, where each $c_i \in \Z$. Then $\sig=0$ if and only if $c_i=c_j$ for all $i$ and $j$ satisfying $i \equiv j \pmod{p^{e-1}}$.
\end{lemma}

\begin{lemma}\label{lem-chavalres}
Let $p$ be a prime and $H$ be a $p$-group. Let $E = \sum_{h\in H}c_h h \in \Z[H]$, where each $c_h \ge 0$ and $\sum_{h\in H}c_h = u$. Suppose there is an integer $\ell$ and a character $\chi \in \wh{H}$ for which $|\chi(E)|^2=\ell$. Then $u^2+(p-1)\ell = pr$ for some integer $r \ge u$.
\end{lemma}

\begin{proof}
Let $p^e = \exp(H)$. Then
\begin{align}
\ell
 &= |\chi(E)|^2 \nonumber \\
 &= \sum_{h, j \in H} c_h c_j \chi(h) \ol{\chi(j)} \nonumber \\
 &= \sum_{i=0}^{p^e-1}d_i\zpe^i, \label{eqn-ellsum}
\end{align}
where $d_i = \sum_{h,j \in H: \chi(h) \ol{\chi(j)} = \zpe^i} c_h c_j$. Each $d_i$ is a non-negative integer, and
\begin{align}
\sum_{i=0}^{p^e-1} d_i
 &= \sum_{h, j \in H} c_h c_j \nonumber \\
 &= u^2. \label{eqn-aisum}
\end{align}
Subtract $\ell$ from both sides of \eqref{eqn-ellsum}, and deduce from Lemma~\ref{lem-rootofunity} that
$$
d_0-\ell=d_{p^{e-1}}=d_{2p^{e-1}}=\cdots=d_{(p-1)p^{e-1}}
$$
and
$$
d_j=d_{p^{e-1}+j}=d_{2p^{e-1}+j}=\cdots=d_{(p-1)p^{e-1}+j}
\quad \mbox{for each $j$ satisfying $1 \le j \le p^{e-1}-1$}.
$$
Substitute into \eqref{eqn-aisum} to obtain
\[
p \sum_{i=0}^{p^{e-1}-1} d_i - (p-1)\ell = u^2,
\]
so that $u^2+(p-1)\ell = pr$ where $r$ is an integer satisfying
\[
r = \sum_{i=0}^{p^{e-1}-1} d_i \ge \, d_0 = \!\!\!\! \sum_{h,j \in H : \chi(h)=\chi(j)}c_hc_j \ge \sum_{h\in H} c_h^2 \ge \sum_{h\in H}c_h =u.
\]
\end{proof}

\begin{theorem}\label{thm-char-scvp}
Suppose $\{D_1,D_2,\ldots,D_m\}$ is a nontrivial $(v,m,k,\la)$-SEDF in a group $G$ with $m>2$, having the simple character value property with respect to $(a,b)$, and let $p$ be a prime divisor of $v$. Then either the following both hold:

\begin{enumerate}
\item[(1a)]
$|G_p|$ divides $km$,
\item[(1b)]
$k^2+(|G_p|-1)\la = |G_p|r_1$ for some integer $r_1 \ge k$.
\end{enumerate}
or the following all hold:
\begin{enumerate}
\item[(2a)]
$k^2m^2+(p-1)\frac{4a^2\la}{b^2-a^2} = pr_2$ for some integer $r_2 \ge km$,
\item[(2b)]
$k^2+(p-1)\frac{(b-a)\la}{b+a} = pr_3$ for some integer $r_3 \ge k$,
\item[(2c)]
$k^2+(p-1)\frac{(b+a)\la}{b-a} = pr_4$ for some integer $r_4$.
\end{enumerate}

\end{theorem}

\begin{proof}
Let $\rho$ be the natural projection from $G$ to $G_p$, and let
$D = \bigcup_{i=1}^m D_i$.
For each nonprincipal character $\wti{\chi} \in \wh{G_p}$ and its associated lifting character~$\chi \in \wh{G}$, Table~\ref{tab-charvalues} gives
\begin{align}
& \big(|\wti{\chi}(\rho(D))|^2,|\wti{\chi}(\rho(D_1))|^2\big) \nonumber \\
& \hspace{0em} = \big(|\chi(D)|^2,|\chi(D_1)|^2\big) =	
\begin{cases}
\bigg(\dfrac{4a^2\la}{b^2-a^2}, \, \dfrac{(b+a)\la}{b-a}\bigg) \mbox{ or }
\bigg(\dfrac{4a^2\la}{b^2-a^2}, \, \dfrac{(b-a)\la}{b+a}\bigg)
		& \mbox{for $\chi \in \wh{G}^N$}, \\[2ex]
(0,\la)		& \mbox{for $\chi \in \wh{G}^0$}.
\end{cases} 						\label{eqn-chirhoDD1}
\end{align}

\begin{description}
\item[Case 1: $\wti{\chi}(\rho(D)) =0$ for every nonprincipal character $\wti{\chi} \in \wh{G_p}$.]
Apply Proposition~\ref{prop-fourier} with $A=\rho(D)$ to obtain $\rho(D) = \frac{km}{|G_p|} G_p$, giving~(1a).
By \eqref{eqn-chirhoDD1}, we have
$|\wti{\chi}(\rho(D_1))|^2 =\la$ for every nonprincipal character $\wti{\chi} \in \wh{G_p}$.
Apply Proposition~\ref{prop-fourier} with $A = \rho(D_1)\rho(D_1)^{(-1)} = \sum_{g\in\wh{G_p}} c_g g$ to obtain $c_1 = \frac{1}{|G_p|}\big(k^2+(|G_p|-1)\la\big)$ and note that $c_1 \ge |D_1| = k$, giving~(1b).

\item[Case 2: $\wti{\chi}(\rho(D)) \ne 0$ for some nonprincipal character $\wti{\chi} \in \wh{G_p}$.]
By \eqref{eqn-chirhoDD1}, this $\wti{\chi}$ satisfies $|\wti{\chi}(\rho(D))|^2=\frac{4a^2\la}{b^2-a^2}$, which is an integer by Lemma~\ref{lem-abpair}~(3).
Apply Lemma~\ref{lem-chavalres} with $(H,E) = (G_p,\rho(D))$ and $u = |D| = km$ to give~(2a).
By \eqref{eqn-chirhoDD1}, this $\wti{\chi}$ also satisfies
$|\wti{\chi}(\rho(D_1))|^2 = \frac{(b+a)\la}{b-a}$ or $\frac{(b-a)\la}{b+a}$.
Then by \eqref{eqn-chiDjsqsim} there is some $j \ne 1$ for which
$$
\big\{ |\wti{\chi}(\rho(D_1))|^2, |\wti{\chi}(\rho(D_j))|^2 \big\}=
\bigg\{ \frac{(b+a)\la}{b-a}, \frac{(b-a)\la}{b+a} \bigg\},
$$
and both values are integers by Lemma~\ref{lem-abpair}~(2).
Apply Lemma~\ref{lem-chavalres} with $(H,E)=(G_p,\rho(D_1))$ and with $(H,E)=(G_p,\rho(D_j))$ to give (2b) and~(2c).
\end{description}
\end{proof}

By Proposition~\ref{prop-nonexistence}~(2), we know that a nontrivial $(v,m,k,\la)$-SEDF does not exist when $v$ is prime and $m>2$.
We now prove a nonexistence result when $v$ is a prime power and $m>2$.

\begin{theorem}\label{thm-primepower-scvp}

Let $G$ be a group of order $v = p^s$ where $p$ is an odd prime, and suppose that $2$ is self-conjugate modulo~$\exp(G)$. Suppose $\{D_1,D_2,\ldots,D_m\}$ is a nontrivial $(v,m,k,\la)$-SEDF in $G$ with $m>2$, having the simple character value property with respect to $(a,b)$. Then $a$ and $b$ are both odd.
\end{theorem}

\begin{proof}
Suppose, for a contradiction, that $a$ and $b$ are not both odd. Since $\gcd(a,b)=1$, we therefore have $b+a$ odd and so $(b^2-a^2) \mid \la$ by Lemma~\ref{lem-abpair}~(3).
We shall show that this implies $km$ and $v-km$ are both even, contradicting that $v = p^s$ for odd~$p$.

Write $D=\bigcup_{i=1}^mD_i$ and use Table~\ref{tab-charvalues} to give
\begin{equation}\label{eqn-tab51scvp}
|\chi(D)|^2 = \frac{4a^2\la}{b^2-a^2} 	\quad 
		\mbox{for all nonprincipal $\chi \in \wh{G}^N$}.
\end{equation}
Since $(b^2-a^2) \mid \la$, this gives
\[
\chi(D)\ol{\chi(D)} \equiv 0 \pmod{2^2} \quad \mbox{for all nonprincipal $\chi \in \wh{G}^N$}.
\]
Since $2$ is self-conjugate modulo~$\exp(G)$, by \cite[Chapter VI, Lemma 13.2]{BJL}, we have
\begin{equation}\label{eqn-even}
\chi(D) \equiv 0 \pmod{2} \quad \mbox{for all nonprincipal $\chi \in \wh{G}^N$}.
\end{equation}

By taking a translate of $D$ if necessary, we may assume that $1 \notin D$.
Write $D=\sum_{g \in G} d_gg \in \Z[G]$. From Proposition~\ref{prop-fourier},
\begin{equation}\label{eqn-0dh}
0=vd_{1}=\sum_{\chi \in \wh{G}} \chi(D) = km+\sum_{\chi \in \wh{G}^N} \chi(D).
\end{equation}
Combining \eqref{eqn-even} and \eqref{eqn-0dh}, we find that $km$ is even.

To show that $v-km$ is even, repeat the above analysis with $D$ replaced by $G\sm D$, noting that $|\chi(G-D)|^2=|\chi(D)|^2$ for each nonprincipal $\chi \in \wh{G}$.
\end{proof}

We now illustrate the use of Theorem~\ref{thm-primepower-scvp} to rule out the existence of an $(81,6,12,9)$-SEDF and a $(6561,6,984,738)$-SEDF.

\begin{example}
Suppose, for a contradiction, that there exists an $(81,6,12,9)$-SEDF or there exists a $(6561,6,984,738)$-SEDF.
By Corollary~\ref{cor-scvp}~(2), these SEDFs have the simple character value property with respect to $(1,2)$.
Since $2$ is self-conjugate modulo $81$ and modulo $6561$, Theorem~\ref{thm-primepower-scvp} then gives the contradiction that $2$ is odd.
\end{example}

\subsection{Further nonexistence results}
In this subsection, we extend the analysis of Section~\ref{subsec-scvp} to the case of an SEDF for which the simple character value property does not necessarily hold.
Suppose that $\{D_1,D_2,\ldots,D_m\}$ is a nontrivial $(v,m,k,\la)$-SEDF in a group $G$ with $m>2$, and let $D = \bigcup_{i=1}^m D_i$. Suppose that $(\ac,\bc)$ as defined in \eqref{eqn-bcac} takes exactly $t\ge1$ distinct values as $\chi$ ranges over $\wh{G}^N$, so that
\[
\Big\{\sqrt{1+\tfrac{4\la}{|\chi(D)|^2}} \mid \chi \in \wh{G}^N \Big\} = \Big\{\frac{b_i}{a_i}\mid 1\le i \le t\Big\},
\]
where $a_i, b_i \in \Z$ and $b_i>a_i>0$ and $\gcd(a_i,b_i)=1$. Define
\[
S^+ = \bigg\{\Big(\frac{4a_i^2\la}{b_i^2-a_i^2},\, \frac{(b_i+a_i)\la}{b_i-a_i}\Big) \mid 1 \le i \le t\bigg\}, \quad \quad
S^- = \bigg\{\Big(\frac{4a_i^2\la}{b_i^2-a_i^2},\, \frac{(b_i-a_i)\la}{b_i+a_i}\Big) \mid 1 \le i \le t\bigg\}.
\]
Then from \eqref{eqn-chiDlaiff}, \eqref{eqn-G0defn} and \eqref{eqn-chiDjsqsim} we have
\begin{equation}\label{eqn-chiDD1t}
\begin{cases}
\big(|\chi(D)|^2, |\chi(D_1)|^2\big) \in S^+ \cup S^- 		
				& \mbox{for $\chi \in \wh{G}^N$}, \\
\big(|\chi(D)|^2, |\chi(D_1)|^2\big) = (0,\la)
				& \mbox{for $\chi \in \wh{G}^0$}.
\end{cases}
\end{equation}

The following result generalizes Theorem~\ref{thm-char-scvp}. The proof, which is omitted, uses \eqref{eqn-chiDD1t} in a similar manner to the use of \eqref{eqn-chirhoDD1} in the proof of Theorem~\ref{thm-char-scvp}.

\begin{theorem}\label{thm-char}
Suppose $\{D_1,D_2,\ldots,D_m\}$ is a nontrivial $(v,m,k,\la)$-SEDF in a group $G$ with $m>2$, let $p$ be a prime divisor of $v$.
Suppose that $(\ac,\bc)$ as defined in \eqref{eqn-bcac} takes values in the set $\{(a_i,b_i) \mid 1 \le i \le t\}$ as $\chi$ ranges over $\wh{G}^N$ (defined with reference to the set $D = \bigcup_{i=1}^m D_i$).
Then either the following both hold:

\begin{enumerate}
\item[(1a)]
$|G_p|$ divides $km$,
\item[(1b)]
$k^2+(|G_p|-1)\la = |G_p|r_1$ for some integer $r_1 \ge k$.
\end{enumerate}
or, for some $i$ satisfying $1 \le i \le t$, the following all hold:
\begin{enumerate}
\item[(2a)]
$k^2m^2+(p-1)\frac{4a_i^2\la}{b_i^2-a_i^2} = pr_{2,i}$ for some integer $r_{2,i} \ge km$,
\item[(2b)]
$k^2+(p-1)\frac{(b_i-a_i)\la}{b_i+a_i} = pr_{3,i}$ for some integer $r_{3,i} \ge k$,
\item[(2c)]
$k^2+(p-1)\frac{(b_i+a_i)\la}{b_i-a_i} = pr_{4,i}$ for some integer $r_{4,i}$.
\end{enumerate}

\end{theorem}

We now illustrate the use of Theorem~\ref{thm-char} to rule out the existence of a (676,26,18,12)-SEDF and a (2401,37,60,54)-SEDF.

\begin{example}
Suppose, for a contradiction, that there exists a $(676,26,18,12)$-SEDF.
By Lemma~\ref{lem-abpair},
$$
(\ac,\bc) \in \{ (1,2),(1,3) \} \quad \mbox{for all $\chi \in \wh{G}^N$}.
$$
For each of these possible values of $(\ac,\bc)$,
both (1a) and (2a) of Theorem~\ref{thm-char} fail with $p=13$, giving the required contradiction.
\end{example}

\begin{example}
Suppose, for a contradiction, that there exists a $(2401,37,60,54)$-SEDF.
By Lemma~\ref{lem-abpair},
$$
(\ac,\bc) \in \{ (1,5),(5,7) \} \quad \mbox{for all $\chi \in \wh{G}^N$}.
$$
We cannot have $(\ac,\bc) = (1,5)$ for $\chi \in \wh{G}^N$, otherwise both (1a) and (2a) of Theorem~\ref{thm-char} fail with $p=7$.
Therefore $(\ac,\bc) = (5,7)$ for all $\chi \in \wh{G}^N$, and so the SEDF satisfies the simple character value property with respect to $(5,7)$.
Theorem~\ref{thm-simchar} then gives $|\wh{G}^0|=\frac{9212}{15}$, which contradicts that $|\wh{G}^0|$ is an integer.
\end{example}

We now extend the nonexistence result of Theorem~\ref{thm-primepower-scvp} for $v$ a prime power and $m>2$.

\begin{theorem}\label{thm-primepower}
Let $v=p^s$ for a prime $p$, and let $\ac,\bc$ be defined as in \eqref{eqn-bcac} with reference to the set $D=\bigcup_{i=1}^m D_i$ associated with a nontrivial $(v,m,k,\la)$-SEDF $\{D_1,D_2,\dots,D_m\}$ in a group $G$ with $m>2$. Let
\begin{equation}\label{eqn-TcUc}
T_\chi = \bigg\{\frac{4\ac^2\la}{\bc^2-\ac^2},\, \frac{(\bc+\ac)\la}{\bc-\ac},\, \frac{(\bc-\ac)\la}{\bc+\ac}\bigg\} \quad \mbox{and} \quad
U_\chi= \begin{cases}	T_\chi			& \mbox{if $|\wh{G}^0| = 0$}, \\
		  	T_\chi \cup \{\la\}	& \mbox{if $|\wh{G}^0|>0$}.
    	\end{cases}
\end{equation}
For each $u \in U_\chi$, if $q$ is a prime divisor of $u$ and $q$ is a primitive root modulo $p^s$, then $q^f \mid\mid u$ for some even $f$.
\end{theorem}
\begin{proof}
For each $u \in U_\chi$, by \eqref{eqn-chiDlaiff}, \eqref{eqn-G0defn} and \eqref{eqn-chiDjsqsim} there is a subset $E_\chi$ of $G$ for which $|\chi(E_\chi)|^2 = u$.
Since $|G|=p^s$, we have $\exp(G)=p^e$ for some integer~$e \le s$ and so $ \chi(E_\chi) \in \Z[\ze_{p^e}]$.
Now if $q$ is a primitive root modulo $p^s$, then $q$ is a primitive root modulo $p^e$. Apply Lemma~\ref{lem-XX'} with $X = \chi(E_\chi)$.

\end{proof}

We now illustrate the use of Theorem~\ref{thm-primepower} to rule out the existence of a $(6561,42,120,90)$-SEDF.

\begin{example}
Suppose, for a contradiction, that there exists a $(6561,42,120,90)$-SEDF.
By Lemma~\ref{lem-abpair},
\[
(\ac,\bc) \in \{(1,2),(1,4),(1,5),(4,5)\} \quad \mbox{for all $\chi \in \wh{G}^N$}.
\]
Since $5$ is a primitive root modulo $6561$, by Theorem~\ref{thm-primepower} we cannot have $(\ac,\bc) \in \{(1,2),(1,5),(4,5)\}$ otherwise the set $T_\chi$ defined in \eqref{eqn-TcUc} contains an element $u$ for which $5 \mid \mid u$.
Therefore $(\ac,\bc) = (1,4)$ for all $\chi \in \wh{G}^N$, and so the SEDF satisfies the simple character value property with respect to $(1,4)$.
Since $2$ is self-conjugate modulo $6561$, Theorem~\ref{thm-primepower-scvp} then gives the contradiction that $4$ is odd.
\end{example}

\begin{remark}\label{rem-open}
Combination of the nonexistence results of
Proposition~\ref{prop-nonexistence}~(4),
Lemmas~\ref{lem-trivial},~\ref{lem-abpair}, and
Theorems~\ref{thm-exp},~\ref{thm-simchar},~\ref{thm-char-scvp},~\ref{thm-primepower-scvp},~\ref{thm-char},~\ref{thm-primepower}, shows that there is no nontrivial $(v,m,k,\la)$-SEDF for $v \le 10^5$ and $m \in \{5,6\}$; and that for $v \le 10^4$ and $m>2$ there are only $70$ possible parameter sets for a nontrivial $(v,m,k,\la)$-SEDF that is not near-complete, namely:
\begin{align*}
 \{&( 540, 12, 42, 36 ),( 784, 30, 18, 12 ),( 1089, 35, 24, 18 ),( 1540, 77, 18, 16 ),( 1701, 35, 30, 18 ),\\
    &( 1701, 35, 40, 32 ),( 2058, 86, 22, 20 ),( 2376, 11, 190, 152 ),( 2401, 7, 280, 196 ),( 2401, 9, 60, 12 ), \\
    &( 2401, 9, 120, 48 ),( 2401, 9, 180, 108 ),( 2401, 9, 240, 192 ),( 2401, 16, 120, 90 ),( 2401, 37, 40, 24 ),\\
    &( 2401, 65, 30, 24 ),( 2500, 18, 105, 75 ),( 2500, 35, 42, 24 ),( 2500, 52, 42, 36 ),( 2601, 53, 40, 32 ), \\
    &( 2625, 42, 48, 36 ),( 2646, 16, 138, 108 ),( 2784, 116, 22, 20 ),( 3025, 57, 36, 24 ),( 3381, 23, 130, 110 ), \\
    &( 3888, 24, 156, 144 ),( 3888, 47, 52, 32 ),( 3888, 47, 78, 72 ),( 3969, 32, 112, 98 ),( 4096, 8, 390, 260 ), \\
    &( 4096, 14, 105, 35 ),( 4096, 14, 210, 140 ),( 4225, 67, 48, 36 ),( 4375, 7, 162, 36 ),( 4375, 7, 324, 144 ),\\
    &( 4375, 7, 486, 324 ),( 4375, 7, 540, 400 ),( 4375, 9, 405, 300 ),( 4375, 16, 270, 250 ),( 4375, 37, 108, 96 ), \\
    &( 4375, 37, 54, 24 ),( 4375, 37, 81, 54 ),( 4564, 163, 26, 24 ),( 4625, 37, 68, 36 ),( 5376, 44, 75, 45 ), \\
    &( 5376, 44, 100, 80 ),( 5776, 78, 60, 48 ),( 5832, 8, 595, 425 ),( 5832, 8, 714, 612 ),( 5832, 18, 147, 63 ), \\
    &( 5832, 18, 294, 252 ),( 5832, 35, 98, 56 ),( 5832, 86, 49, 35 ),( 5888, 92, 58, 52 ),( 6400, 80, 54, 36 ), \\
    &( 6656, 26, 121, 55 ),( 6656, 26, 242, 220 ),( 6860, 20, 266, 196 ),( 6860, 58, 95, 75 ),( 6976, 218, 30, 28 ), \\
    &( 8281, 93, 60, 40 ),( 8625, 23, 140, 50 ),( 8625, 23, 280, 200 ),( 8960, 7, 1054, 744 ),( 8960, 32, 238, 196 ), \\
    &( 9801, 13, 420, 216 ),( 9801, 26, 308, 242 ),( 9801, 57, 140, 112 ),( 9801, 101, 70, 50 ),( 9801, 101, 84, 72 ) \}.
\end{align*}
\end{remark}

In Section~\ref{sec-exp-bound} we proved the exponent bound of Theorem~\ref{thm-exp} using only information about the SEDF parameters $(v,m,k,\la)$, and applied it to the case $m=2$. We now derive a different exponent bound that uses information about the possible values of $|\chi(D)|^2$ and $|\chi(D_1)|^2$, and apply it to two of the open cases with $m>2$ given in Remark~\ref{rem-open}.

\begin{theorem}\label{thm-exp2}
Suppose $\{D_1,D_2,\ldots,D_m\}$ is a $(v,m,k,\la)$-SEDF in a group~$G$, let $D = \bigcup_{i=1}^m D_i$, and let $p$ be a prime dividing $v$.
Suppose $U$ is a subgroup of $G$ for which $U \cap G_p = \{1\}$ and $p$ is self-conjugate modulo $\exp(G/U)$.
\begin{enumerate}[(1)]
\item If $|\chi(D)|^2 \equiv0 \pmod{p^{2d}}$ for every nonprincipal $\chi \in \wh{G}$, then \\
$\exp(G_p)\le \max\bigg\{ \dfrac{|U|}{p^d}|G_p|,\dfrac{p\,|\wh{G}^0|\cdot |U|}{(p-1)v}|G_p|\bigg\}$.
\item If $|\chi(D_1)|^2 \equiv0 \pmod{p^{2d}}$ for every nonprincipal $\chi \in \wh{G}$, then $\exp(G_p)\le \dfrac{|U|}{p^d}|G_p|$.
\end{enumerate}
\end{theorem}
\begin{proof}
The proof is analogous to that of \cite[Chapter VI, Theorem 15.11]{BJL}. We prove only (1); the proof of (2) is similar.

Let $W$ be a subgroup of $G_p$ for which $G_p/W$ is cyclic of order $\exp(G_p)$.
It follows from $U \cap G_p= \{1\}$ that $U \cap W=\{1\}$, and so we may write $H = U \times W$. Since $\exp(G_p/W)=\exp(G_p)$, we then have $\exp(G/H)=\exp(G/U)$. Let $\rho$ be the canonical epimorphism $\rho: G \rightarrow G/H$. Then by assumption,
$$
|\wti{\chi}(\rho(D))|^2 \equiv 0 \pmod{p^{2d}}
\quad \mbox{for every nonprincipal $\wti{\chi}\in\wh{G/H}$}.
$$
Since $p$ is self-conjugate modulo $\exp(G/H)$, this implies \cite[Chapter VI, Lemma 13.2]{BJL}
\[
\wti{\chi}(\rho(D)) \equiv 0 \pmod{p^d}
\quad \mbox{for every nonprincipal $\wti{\chi}\in\wh{G/H}$},
\]
and then by Ma's Lemma \cite{Ma} we have
$$
\rho(D)=p^dX_0+PX_1,
$$
where $X_0,X_1 \in \Z[G/H]$ have non-negative coefficients and $P$ is the unique subgroup of $G/H$ of order~$p$.

In the case $X_0 \ne 0$, we have $p^d \le |H|=|U|\cdot|W|=|U|\cdot\frac{|G_p|}{\exp(G_p)}$, which rearranges to $\exp(G_p) \le \frac{|U|}{p^d}|G_p|$.

Otherwise, in the case $X_0=0$, we have $\rho(D)=PX_1$. Now consider the $\frac{|G|}{|H|}(1-\frac{1}{p})$ characters $\wti{\chi} \in \wh{G/H}$ which are nonprincipal on $P$. Each such character satisfies $\wti{\chi}(\rho(D)) = 0$, and its associated lifting character $\chi \in \wh{G}$ satisfies $\chi(D)=0$. Therefore by the definition \eqref{eqn-G0defn} of $\wh{G^0}$, we have $|\wh{G}^0| \ge \frac{|G|}{|H|}(1-\frac{1}{p})$, which implies $\exp(G_p) \le \frac{p\,|\wh{G}^0|\cdot|U|}{(p-1)v}|G_p|$.
\end{proof}

We now illustrate the use of Theorem~\ref{thm-exp2} to obtain an exponent bound on a group containing a $(2401,7,280,196)$-SEDF and a group containing a $(5832,8,595,425)$-SEDF.

\begin{example}
Suppose there exists a $(2401,7,280,196)$-SEDF in a group $G$. Note that $196=2^2\cdot7^2$ and that neither $2$ nor $7$ is a primitive root modulo $7^4=2401$, so Theorem~\ref{thm-exp} does not apply. However, by Lemma~\ref{lem-abpair} the SEDF satisfies the simple character value property with respect to $(a,b)=(3,5)$, and so from Table~\ref{tab-charvalues} we have $|\chi(D_1)|^2 \in \{7^2,\, 4\cdot 7^2,\, 16 \cdot 7^2\}$ for every nonprincipal $\chi \in \wh{G}$.
Since $7$ is self-conjugate modulo $2401$, we may apply Theorem~\ref{thm-exp2}~(2) with $(p,d)=(7,1)$ and $U = \{1\}$ to show that $\exp(G)\le 7^3$.
\end{example}

\begin{example}
Suppose there exists a $(5832,8,595,425)$-SEDF in $G$. Note that $5832 = 2^3 \cdot 3^6$ and $425 = 5^2 \cdot 17$. Theorem~\ref{thm-exp} does not give any constraint on the structure of $G$ (even though it may be applied with $(p,q)=(3,5)$).
By Lemma~\ref{lem-abpair}, the SEDF satisfies the simple character value property with respect to $(2,3)$, and so from Table~\ref{tab-charvalues} we have $|\chi(D)|^2 \in \{0,\, 2^4\cdot5\cdot17\}$ and $|\chi(D_1)|^2 \in \{5\cdot17,\, 5^2\cdot17,\, 5^3\cdot17\}$ for every nonprincipal $\chi \in \wh{G}$, and
$|\wh{G}^0|=2079$ from Theorem~\ref{thm-simchar}.
In this case Theorem~\ref{thm-exp2}~(2) does not apply. However,
because $2$ is self-conjugate modulo $2^3 \cdot 3^6$,
we may apply Theorem~\ref{thm-exp2}~(1) with $(p,d)=(2,2)$ and $U =\{1\}$ to obtain
$\exp(G_2)\le \max{\{2,\frac{154}{27}\}} < 2^3$.
Therefore $\exp(G_2)\le 2^2$.
\end{example}

\section{Concluding remarks}
\label{sec-conc}
We have presented a comprehensive treatment of SEDFs, using character theory and algebraic number theory to derive many nonexistence results. We have characterized the parameters of a nontrivial near-complete SEDF, and constructed a $(243,11,22,20)$-SEDF in $\Z_3^5$ from a detailed analysis of the action of the Mathieu group $M_{11}$ on the points of the projective geometry $PG(4,3)$. This is the first known nontrivial example of SEDF with $m>2$.

As we were finalizing our paper, Wen, Yang and Feng posted a preprint \cite{WYF} in which they independently constructed a $(243,11,22,20)$-SEDF in $\Z_3^5$ using cyclotomic classes over $\F_{3^5}$. Their method, which was also used to construct some generalizations of SEDFs~\cite{WYFF}, is very different from ours.

In closing, we note that until now SEDFs have been considered only in abelian groups. We ask: are there examples of nontrivial SEDFs in nonabelian groups?

\section*{Acknowledgements}
We are grateful to Ruizhong Wei for kindly supplying a preprint of the paper~\cite{BJWZ}. We thank the referee for providing very careful and helpful comments.


\end{document}